\documentclass[12pt,reqno]{amsart}
\usepackage[margin=1in]{geometry}
\usepackage{amssymb,amsfonts,amsmath,mathrsfs,etoolbox,mathtools,bm}
\usepackage[breaklinks=true,colorlinks=true,linkcolor=teal,citecolor=teal,filecolor=teal,urlcolor=black!50!blue]{hyperref}
\hypersetup{linktocpage}
\usepackage[hiresbb]{graphicx,xcolor}
\allowdisplaybreaks[4]

\newtheorem{theorem}{Theorem}[section]
\newtheorem{corollary}[theorem]{Corollary}
\newtheorem{lemma}[theorem]{Lemma}
\newtheorem{problem}[theorem]{Problem}

\theoremstyle{definition}
\newtheorem{remark}[theorem]{Remark}

\numberwithin{equation}{section}

\newcommand{\N}{\mathbb{N}}
\newcommand{\R}{\mathbb{R}}
\newcommand{\Z}{\mathbb{Z}}
\newcommand{\GL}{\mathrm{GL}}
\newcommand{\sumast}{\mathop{\sum\nolimits^{\mathrlap{\ast}}}}

\renewcommand{\Im}{\mathrm{Im}}
\renewcommand{\Re}{\mathrm{Re}}
\renewcommand{\tilde}{\widetilde}
\renewcommand{\epsilon}{\varepsilon}

\patchcmd{\section}{\scshape}{\bfseries}{}{}
\makeatletter
\renewcommand{\@secnumfont}{\bfseries}
\makeatother
\newcommand{\norm}[1]{\left\lVert#1\right\rVert}
\makeatletter\newcommand{\tpmod}[1]{{\@displayfalse \pmod{#1}}}

\begin{document}

\title{Non-Linear Additive Twists of $\mathrm{GL}_{3}$ Hecke Eigenvalues}

\author{Ikuya Kaneko}
\address{The Division of Physics, Mathematics and Astronomy, California Institute of Technology, 1200 E. California Blvd., Pasadena, CA 91125, USA}
\email{ikuyak@icloud.com}
\urladdr{\href{https://sites.google.com/view/ikuyakaneko/}{https://sites.google.com/view/ikuyakaneko/}}

\author{Wing Hong Leung}
\address{Department of Mathematics, Mailstop 3368, Texas A\&M University, College Station, TX 77843-3368, USA}
\email{josephleung102792@gmail.com}
\urladdr{\href{https://sites.google.com/view/winghongleung/}{https://sites.google.com/view/winghongleung/}}

\thanks{IK acknowledges the support of the Masason Foundation.}



\date{\today}

\dedicatory{}

\begin{abstract}
We bound non-linear additive twists of $\mathrm{GL}_{3}$ Hecke eigenvalues, improving upon the work of Kumar--Mallesham--Singh (2022). The proof employs the DFI circle method with standard manipulations (Vorono\u{\i}, Cauchy--Schwarz, lengthening, and additive reciprocity). The main novelty includes the conductor lowering mechanism, albeit sacrificing some savings to remove an analytic oscillation, followed by the iteration \textit{ad infinitum} of Cauchy--Schwarz and Poisson. The resulting character sums are estimated via the work of Adolphson--Sperber (1993). As an application, we prove nontrivial bounds for the first moment of $\GL_{3}$~Hardy's function, which corresponds to the cubic moment of Hardy's function studied by Ivi\'{c} (2012).
\end{abstract}

\maketitle
\tableofcontents

\section{Introduction}\label{introduction}

\subsection{Statement of Main Results}\label{statement-of-main-results}
Let $\mathbb{A}_{\mathbb{Q}}$ be the ring of ad\`{e}les over $\mathbb{Q}$, and let $\mathfrak{F}_{k}$ be the universal family of all cuspidal automorphic representations of $\mathrm{GL}_{k}(\mathbb{A}_{\mathbb{Q}})$ with unitary~central character, which we normalise to be trivial on the diagonally embedded copy of the positive reals so that $\mathfrak{F}_{k}$ is discrete; see, for example, Brumley~\cite[Corollary~9]{Brumley2006}. If $(a_{n})$ represents a sequence of arbitrary complex numbers, then a well-known problem in number theory is~to estimate nontrivially its correlations against \emph{non-linear additive twists} of the shape\footnote{One often considers a dyadically smoothed version of~\eqref{non-linear} by inserting a smooth weight function therein, but it is aesthetically convenient in the present paper to deal with the unsmoothed version~\eqref{non-linear}.}
\begin{equation}\label{non-linear}
\sum_{n \leq T} a_{n} e(\alpha n^{\beta}),
\end{equation}
where $T \geq 1$ and $\alpha, \beta \in \R$. The summand is called an additive twist of $(a_{n})$ when $\beta = 1$. Of particular interest is when $(a_{n})$ has an automorphic origin due to its relevance to moments of families of $L$-functions and subconvex bounds thereof. For example, if $(a_{n})$ arises from the Hecke eigenvalues of $\varphi \in \mathfrak{F}_{2}$, then the study of~\eqref{non-linear} has a distinguished history. In particular, Jutila~\cite{Jutila1987} improves upon a \emph{Wilton-type bound} and shows uniformly in $\alpha \in \R$~that
\begin{equation}\label{Jutila}
\sum_{n \leq T} \lambda_{\varphi}(n) e(\alpha n) \ll \sqrt{T},
\end{equation}
where $\lambda_{\varphi}(n)$ is the $n$-th Hecke eigenvalue of $\varphi$. 

Given $\pi \in \mathfrak{F}_{k}$, emblematic generalisations of~\eqref{Jutila} include
\begin{equation}\label{additive-twists}
\sum_{n \leq T} \lambda_{\pi}(n) e(\alpha n),
\end{equation}
where $\lambda_{\pi}(n) \coloneqq \lambda_{\pi}(1, n)$ is the $n$-th Hecke eigenvalue of $\pi$. For $k = 3$, the best known bound of Miller~\cite[Theorem 1.1]{Miller2006} goes halfway between the trivial and optimal bound, namely
\begin{equation}\label{Miller}
\sum_{n \leq T} \lambda_{\pi}(n) e(\alpha n) \ll_{\pi, \epsilon} T^{\frac{3}{4}+\epsilon}
\end{equation}
uniformly in $\alpha \in \R$. For $k \geq 4$, nontrivial bounds for~\eqref{additive-twists} exhibit a hindrance of progress~at the current state of knowledge; see~\cite{RenYe2015-2,RenYe2015} for further details.

If $\pi \in \mathfrak{F}_{3}$ and $\alpha, \beta > 0$, then the result of Ren--Ye~\cite[Theorem~2]{RenYe2015} states\footnote{There is a typographical error in the exponent of $\alpha$ in~\cite[Equation (3)]{KumarMalleshamSingh2022}.} that
\begin{equation*}
\mathcal{A}_{\pi}(T, \alpha, \beta) \coloneqq \sum_{n \leq T} \lambda_{\pi}(n) e(\alpha n^{\beta}) \ll_{\pi, \beta} (T^{\frac{2}{3}}+(\alpha T^{\beta})^{\frac{3}{2}}) \log T
\end{equation*}
for $\beta \ne \frac{1}{3}$, and
\begin{equation*}
\mathcal{A}_{\pi}(T, \alpha, \beta) \ll_{\pi, \epsilon} (1+\alpha^{\frac{1}{14}+\epsilon})(T^{\frac{2}{3}} \log T+\alpha^{2} T^{\frac{1}{3}+\epsilon})
\end{equation*}
for $\beta = \frac{1}{3}$. Furthermore, Kumar--Mallesham--Singh~\cite[Theorem 1]{KumarMalleshamSingh2022} prove
\begin{equation}\label{eq:alpha-beta}
\mathcal{A}_{\pi}(T, \alpha, \beta) \ll_{\pi,  \epsilon} \alpha \sqrt{\beta} T^{\frac{3}{4}+\frac{9\beta}{28}+\epsilon}
\end{equation}
for $\frac{1}{3} < \beta < \frac{7}{9}$. In particular, for fixed $\alpha \in \R$,~\eqref{eq:alpha-beta} implies
\begin{equation*}
\mathcal{A}_{\pi}(T, \alpha) \coloneqq \mathcal{A}_{\pi} \left(T, \alpha, \frac{2}{3} \right)
 = \sum_{n \leq T} \lambda_{\pi}(n) e(\alpha n^{\frac{2}{3}}) \ll_{\pi, \epsilon} T^{\frac{27}{28}+\epsilon}.
\end{equation*}
This work provides a substantial improvement over~\eqref{eq:alpha-beta}. We regard $\alpha$ and $\beta$ as fixed, but~our method generalises without any problems to let them vary in their appropriate regimes.
\begin{theorem}\label{main}
Let $\pi \in \mathfrak{F}_{3}$, and let $\alpha \in \R$, $\frac{5}{14} < \beta < \frac{4}{5}$. Then we have for any $\epsilon > 0$ that
\begin{equation*}
\mathcal{A}_{\pi}(T, \alpha, \beta) \ll_{\pi, \alpha, \beta, \epsilon} 
T^{\frac{2}{3}+\frac{5\beta}{12}+\epsilon}+T^{\frac{7}{8}+\varepsilon}+T^{\frac{19}{14}-\beta+\epsilon}.
\end{equation*}
\end{theorem}

Theorem~\ref{main} improves upon~\cite[Theorem~1]{KumarMalleshamSingh2022} for $\frac{17}{37} < \beta < \frac{4}{5}$. The proof is predicated upon the Duke--Friedlander--Iwaniec circle method together with the fairly standard toolbox including Vorono\u{\i}, Poisson, Cauchy--Schwarz, lengthening, and additive reciprocity.~We use conductor lowering, but sacrifice some savings to eliminate a cumbersome analytic oscillation after Vorono\u{\i}. Another key input is an iteration \textit{ad infinitum} of Cauchy--Schwarz and Poisson summation to attain additional savings. We highlight that this iteration gives rise to certain intricate character sums that are viewed as a simpler analogue of what is handled in~\cite{AggarwalLeungMunshi2022}. This is where a result of Adolphson--Sperber~\cite{AdolphsonSperber1993} comes into play. They extended Dwork's cohomology theory from smooth projective hypersurfaces in characteristic $p$ to a sufficiently wide class of exponential sums.

\begin{corollary}\label{corollary}
Let $\pi \in \mathfrak{F}_{3}$, and let $\alpha \in \R$. Then we have for any $\epsilon > 0$ that
\begin{equation*}
\mathcal{A}_{\pi}(T, \alpha) \ll_{\pi, \alpha, \epsilon} T^{\frac{17}{18}+\epsilon}.
\end{equation*}
\end{corollary}

Corollary~\ref{corollary} has an application to the first moment of $\GL_{3}$ Hardy's function; see Section~\ref{applications}.

We now formulate a smoothed version of Theorem~\ref{main}. Let $1\leq Y\ll T^{\beta-\varepsilon}$ be a parameter at our disposal, and let $W$ be a smooth function supported on $[1,2]$ such that $x^jW^{(j)}(x) \ll_{j} Y^j$ for all fixed $j \in \N_{0}$.
\begin{theorem}\label{main-2}
Keep the notation and assumptions as above.  Then we have for any $\epsilon > 0$ that
\begin{equation*}
\mathcal{A}_{\pi}^{Y}(T, \alpha, \beta):=\sum_{n = 1}^{\infty} \lambda_{\pi}(n) e(\alpha n^{\beta}) W \left(\frac{n}{T} \right) 
\ll_{\pi, \alpha, \beta, \epsilon} T^{\frac{2}{3}+\frac{5\beta}{12}+\epsilon}+T^{\frac{7}{8}+\varepsilon}.
\end{equation*}
\end{theorem}

\subsection{Heuristic Sketch}\label{heuristic-sketch}
This section overviews the method of the proof of Theorem~\ref{main} in a back-of-the-envelope fashion, presenting a high-level sketch geared to experts. This~section is not intended to entail rigour but might serve as a roadmap. In particular, we ignore some subtleties such as complicated smooth weights, coprimality conditions, and common divisors.

\subsubsection{Trivial Bound}
By Cauchy--Schwarz and the Rankin--Selberg bound (also known as the Ramanujan conjecture on average), the trivial bound amounts to
\begin{equation*}
\mathcal{A}_{\pi}(T, \alpha, \beta) \ll_{\pi} T
\end{equation*}
uniformly in $\alpha, \beta \in \R$. In what follows, we shall implicitly employ the Rankin--Selberg bound to visualise how much we ought to save from the trivial bound at each step.

\subsubsection{Smoothing}
Via a standard smoothing argument, the proof of Theorem \ref{main} reduces~to the proof of Theorem \ref{main-2}, because roughly
\begin{equation}\label{SketchSmoothing}
    \mathcal{A}_{\pi}(T, \alpha, \beta) = \mathcal{A}_{\pi}^{Y}(T, \alpha, \beta)+O\left(\frac{T^{\frac{19}{14}+\varepsilon}}{Y} \right).
\end{equation}

\subsubsection{Applying the $\delta$-Symbol}
The Duke--Friedlander--Iwaniec circle method now implies\footnote{Here and henceforth, we employ the symbol $\approx$ to indicate that the left-hand side may roughly be written as an expression resembling the right-hand side with an admissible error term. Furthermore, we replace the condition $n \leq T$ with $n \sim T$ for brevity, and this tentative simplification applies to other summations.}
\begin{align*}
\mathcal{A}_{\pi}^{Y}(T, \alpha, \beta) &\approx \sum_{m \sim T} \lambda_{\pi}(m) \sum_{n \sim T} e(\alpha n^{\beta})W\left(\frac{n}{T}\right) \delta(m = n)\\
&\approx \frac{1}{Q^2} \sum_{q \sim Q} \ \sideset{}{^{\ast}} \sum_{a \tpmod{q}} 
\sum_{m \sim T} \lambda_{\pi}(m) e \left(\frac{am}{q} \right) 
\sum_{n \sim T} e(\alpha n^{\beta}) e \left(-\frac{an}{q} \right)W\left(\frac{n}{T}\right)U\left(\frac{m-n}{qQ}\right)
\end{align*}
for any parameter $Q\geq1$ and some fixed function $U\in C_c^\infty([-1,1])$. The trivial bound at~this stage is $O_{\pi}(T^{2})$, and we need to save $T$ and a little more.

\subsubsection{Vorono\u{\i} Summation in $m$}
As an initial manoeuvre, we use $\GL_{3}$ Vorono\u{\i} summation in the sum over $m$. A sophisticated analysis of the integral transform on the dual side yields
\begin{equation*}
\sum_{m \sim T} \lambda_{\pi}(m) e \left(\frac{am}{q} \right) U\left(\frac{m-n}{qQ}\right)
\approx \frac{Q^2}{T}\sum_{m \sim \frac{T^2}{Q^3}} \lambda_{\pi}(m) S(\overline{a}, m; q)e\left(\frac{3(mn)^{\frac{1}{3}}}{q}\right),
\end{equation*}
where $S(a, b; c)$ denotes the standard Kloosterman sum.

\subsubsection{Poisson Summation in $n$}
We apply Poisson summation to the sum over $n$. Choosing $Q\gg T^{\frac{1-\beta}{2}+\varepsilon}$ yields something like
\begin{align*}
\sum_{n \sim T} e\left(-\frac{an}{q} +\alpha n^{\beta}+\frac{3(mn)^{\frac{1}{3}}}{q}\right) W\left(\frac{n}{T}\right) \approx T \sum_{\substack{n \sim \frac{Q}{T^{1-\beta}} \\ n \equiv a \tpmod{q}}} \mathcal{I}(m,n,q),
\end{align*}
where \begin{align*}
    \mathcal{I}(m,n,q) \coloneqq \int_0^\infty W(y)e\left(\alpha(Ty)^\beta+\frac{3(mTy)^{\frac{1}{3}}}{q}-\frac{nTy}{q}\right)dy.
\end{align*}
Hence, there holds
\begin{equation*}
\mathcal{A}_{\pi}^{Y}(T, \alpha, \beta) \approx  \sum_{m \sim \frac{T^2}{Q^3}} 
\sum_{n \sim \frac{Q}{T^{1-\beta}}} \sum_{q \sim Q} \lambda_{\pi}(m) 
S(m, \overline{n}; q) \mathcal{I}(m,n,q).
\end{equation*}
A stationary phase analysis yields $\mathcal{I}(m,n,q) \ll_{\epsilon} T^{-\frac{\beta}{2}+\varepsilon}$, thus the trivial bound at this stage is $O_{\pi, \alpha, \beta}(Q^{-\frac{1}{2}} T^{1+\frac{\beta}{2}})$.

\subsubsection{Rewriting the Integral}
It is now convenient to rewrite the oscillatory factor $e(\frac{3(mTy)^{\frac{1}{3}}}{q})$ via Mellin inversion. In fact, one approximates
\begin{equation*}
    e(x) \approx \frac{1}{\sqrt{x}} \int_{t\sim x}f(t)x^{it} dt
\end{equation*}
for some appropriate function $f\ll 1$. This allows one to separate $m$ from $q$ in the ensuing Cauchy--Schwarz step at the sacrifice of losing a square-root saving from the above integral.

\subsubsection{Cauchy--Schwarz and Lengthening in $m$}
Cauchy--Schwarz implies
\begin{equation}\label{eq:Cauchy-Schwarz}
\mathcal{A}_{\pi}^{Y}(T, \alpha, \beta) \ll_{\pi, \alpha, \beta} \frac{T^{\frac{3}{2}}}{Q^{\frac{5}{2}}} \sup_{t} \left(\sum_{m \sim \frac{T^2}{Q^3}} 
\left|\sum_{n \sim \frac{Q}{T^{1-\beta}}} \sum_{q \sim Q} S(m, \overline{n}; q) 
\mathcal{J}(n,q) \right|^{2} \right)^{\frac{1}{2}},
\end{equation}
where \begin{align*}
    \mathcal{J}(n,q) \coloneqq q^{-it}\int_0^\infty W(y)e\left(\alpha(Ty)^\beta-\frac{nTy}{q}\right)dy \ll T^{-\beta/2+\varepsilon}.
\end{align*}
The square-root cancellation heuristic now demonstrates that the best possible bound for~the expression~\eqref{eq:Cauchy-Schwarz} is $O_{\pi, \alpha, \beta, \epsilon}(Q^{-\frac{5}{2}} T^{2+\epsilon})$. The expression inside the square root in~\eqref{eq:Cauchy-Schwarz} becomes
\begin{equation*}
\sum_{m \sim \frac{T^2}{Q^3}} \mathop{\sum\sum}_{n_{1},n_2 \sim \frac{Q}{T^{1-\beta}}}  
\mathop{\sum\sum}_{q_{1}, q_2 \sim Q} S(m, \overline{n_{1}}; q_{1}) S(m, \overline{n_{2}}; q_{2}) 
\mathcal{J}(n_1,q_1)\overline{\mathcal{J}(n_2,q_2)}.
\end{equation*}
Rigorously, we introduce an lengthening parameter $L \geq 1$ and make the sum over $m$ of~length $\frac{LT^{2}}{Q^{3}}$ whose purpose is to rebalance the contributions of zero frequency and nonzero frequencies in Poisson summation in some non-generic cases. Since we choose $L \asymp 1$ in the generic case, we here refrain from discussing this step in detail.

\subsubsection{Poisson Summation in $m$}
We are in a position to execute Poisson summation in~$m$, obtaining something like
\begin{multline}\label{eq:Poisson-m-rough}
\sum_{m \sim \frac{T^2}{Q^3}} S(m, \overline{n_{1}}; q_{1}) S(m, \overline{n_{2}}; q_{2}) 
\approx \frac{T^2}{Q^2} \cdot \delta(n_{1} = n_{2}) \delta(q_{1} = q_{2})\\
 + \frac{T^2}{Q^5} \sum_{m \sim \frac{Q^5}{T^2}} \sum_{\gamma \tpmod{q_{1} q_{2}}} 
S(\gamma, \overline{n_{1}}; q_{1}) S(\gamma, \overline{n_{2}}; q_{2}) e \left(\frac{\gamma m}{q_{1} q_{2}} \right).
\end{multline}

\subsubsection{Additive Reciprocity}
By twisted multiplicativity for Kloosterman sums and additive reciprocity, the sum over $\gamma$ in~\eqref{eq:Poisson-m-rough} transforms into (ignoring an archimedean exponential)
\begin{equation*}
q_{1} q_{2} \cdot e \left(\frac{\overline{mn_{1}} q_{2}}{q_{1}}+\frac{\overline{mn_{2}} q_{1}}{q_{2}} \right) 
\approx q_{1} q_{2} \cdot e \left(\frac{\overline{q_{1}} q_{2}}{mn_{1}}+\frac{q_{1} \overline{q_{2}}}{mn_{2}} \right).
\end{equation*}
Hence, there holds
\begin{equation}\label{SketchBound1}
\mathcal{A}_{\pi}^{Y}(T, \alpha, \beta) \ll_{\pi, \alpha, \beta, \epsilon} Q^{-\frac{5}{2}} T^{2+\epsilon}+\Omega,
\end{equation}
where the contribution of nonzero frequencies is defined by
\begin{equation*}
\Omega \coloneqq  \frac{T^{\frac{5}{2}}}{Q^4} \left(\sum_{m \sim \frac{Q^5}{T^2}} 
\mathop{\sum\sum}_{n_{1}, n_2 \sim \frac{Q}{T^{1-\beta}}}  \mathop{\sum\sum}_{q_{1},q_2 \sim Q} e \left(\frac{\overline{q_{1}} q_{2}}{mn_{1}}+\frac{q_{1} \overline{q_{2}}}{mn_{2}} \right)\mathcal{J}(n_1,q_1)\overline{\mathcal{J}(n_2,q_2)} \right)^{\frac{1}{2}}.
\end{equation*}
The trivial bound for $\Omega$ is $O(\sqrt{Q}T^{\frac{1}{2}+\frac{\beta}{2}})$.

\subsubsection{Cauchy--Schwarz and Poisson Summation in $q_{1}$}
We apply Cauchy--Schwarz to~take out everything except the $q_2$-sum, and introduce a lengthening parameter $R=1+Q^6T^{-4+2\beta}$ to the $q_1$-sum. Upon applying Poisson summation in $q_{1}$, the lengthening forces the dual side to degenerate to frequency zero. While this step does not theoretically affect the final bound, it substantially facilitates the rest of the proof. To summarise, we are led to the expression
\begin{multline*}
    \Omega\ll \frac{T^{\frac{5}{2}}}{Q^4} \left(\frac{RQ^4}{T^{2-\beta}}\right)^{\frac{1}{4}} \Bigg(\sum_{m \sim \frac{Q^5}{T^2}} 
\mathop{\sum\sum}_{n_{1}, n_2 \sim \frac{Q}{T^{1-\beta}}}  \mathop{\sum\sum}_{\substack{q_2,q_3 \sim Q\\ q_2\equiv q_3\tpmod{n_1n_2}}}\\
\times S(\overline{q_2}-\overline{q_3},\overline{n_1n_2}(q_2-q_3);m) \mathcal{J}(n_2,q_2)\overline{\mathcal{J}(n_2,q_3)}\Bigg)^{\frac{1}{4}}.
\end{multline*}
The trivial bound for $\Omega$ at this stage is $O(Q^{-\frac{5}{8}} R^{\frac{1}{4}} T^{\frac{5}{4}})$.

\subsubsection{Iteration Ad Infinitum}
One observes that the dual side preserves the structure
\begin{align*}
    \mathop{\sum\sum}_{q_2,q_3\sim Q} \mathcal{C}(q_2,q_3)
\end{align*}
for a more intricate character sum $\mathcal{C}(q_2,q_3)$ defined modulo $mn_1n_2$. This enables the iteration of the above procedure. Iterating $k$ times for sufficiently large $k$, the total saving amounts~to
\begin{align*}
    \left(\frac{Q^{\frac{5}{2}}}{RT}\right)^{\frac{1}{8}} \times \left(\frac{Q^{\frac{5}{2}}}{RT}\right)^{\frac{1}{16}} \times\cdots \times \left(\frac{Q^{\frac{5}{2}}}{RT}\right)^{\frac{1}{2^{k}}}\approx \left(\frac{Q^{\frac{5}{2}}}{RT}\right)^{\frac{1}{4}} T^{-\varepsilon}.
\end{align*}
This implies
\begin{align}\label{OmegaBound}
    \Omega\ll \frac{R^{\frac{1}{4}}T^{\frac{5}{4}+\varepsilon}}{Q^{\frac{5}{8}}}\left(\frac{RT}{Q^{\frac{5}{2}}}\right)^{\frac{1}{4}} \ll \frac{T^{\frac{3}{2}+\varepsilon}}{Q^{\frac{5}{4}}}\left(1+\frac{Q^3}{T^{2-\beta}}\right).
\end{align}
Inserting \eqref{OmegaBound} into \eqref{SketchBound1} along with the optimisation $Q=\min\{T^{\frac{2-\beta}{3}},\sqrt{T}\}$ justifies the~claim of Theorem~\ref{main-2}. Theorem~\ref{main} now follows by choosing $Y=T^{\beta-\varepsilon}$ in \eqref{SketchSmoothing}.

\subsection*{Acknowledgements}
The authors thank Keshav Aggarwal for helpful discussions that led to this paper being written.

\section{Conventions}
Throughout, we use the shorthand $e(x) = e^{2\pi ix}$. We make use of the letter $\epsilon > 0$ to denote an arbitrarily small positive quantity that is possibly different in each instance. The notation $f \ll_{\nu} g$ or $f = O_{\nu}(g)$ indicates that there exists an effectively computable constant $c_{\nu} > 0$, depending at most on $\nu$, such that $|f(z)| \leq c_{\nu} |g(z)|$ for all $z$ in a range that is apparent~from context. If no parameter $\nu$ is present, then $c$ is absolute. If $a$ and $b$ are either both positive or both negative, then the notation $a \sim b$ means $k_{1} < \frac{a}{b} < k_{2}$ for some absolute and effectively computable constants $k_{1}, k_{2} > 0$. If they depend on $\nu$, then we attach a subscript and write $a \sim_{\nu} b$. The Kronecker symbol $\delta(\mathrm{S})$ detects $1$ if the statement $\mathrm{S}$ is true, and $0$ otherwise. For any positive integer $n$ with its prime factorisation $n=\prod_{i=1}^kp_i^{a_i}$, we write $n_\square=\prod_{\substack{1\leq i\leq k\\ a_i\geq2}}p_i^{a_i}$.

\section{Applications}\label{applications}
This section describes an application of Theorem~\ref{main} to moments of $\GL_{3}$ Hardy's function.

\subsection{Classical Triumphs}\label{classical-triumphs}
Hardy's function plays a crucial r\^{o}le in the study of the zeros~of the Riemann zeta function $\zeta(s)$ on the critical line $\Re(s) = \frac{1}{2}$. Given $t \in \R$, it is defined by
\begin{equation*}
Z(t) \coloneqq \zeta \left(\frac{1}{2}+it \right) \chi \left(\frac{1}{2}+it \right)^{-\frac{1}{2}}, \qquad \zeta(s) = \chi(s) \zeta(1-s).
\end{equation*}
As Ivi\'{c} discusses in~\cite{Ivic2013}, Hardy's function $Z(t)$ is real-valued, even, and smooth. Moreover, $|Z(t)| = |\zeta(\frac{1}{2}+it)|$ implies that its zeros match the zeros of $\zeta(s)$ on the critical line. Hardy's original application of $Z(t)$ is to prove that $\zeta(s)$ has infinitely many nontrivial zeros; see the book of Titchmarsh~\cite{Titchmarsh1987}. The proof begins with an assumption that $Z(t)$ does not~change sign for sufficiently large $t > 0$ and then yields a contradiction. Hardy~\cite{Hardy1914} later refines~on his reasoning to show that $\zeta(s)$ has $\gg T$ zeros $\beta+i\gamma$ satisfying $|\gamma| \leq T$. Selberg~\cite{Selberg1989} then improves upon this estimate to $\gg T \log T$, followed by its further refinements by successors.

We are interested in the moments of $Z(t)$, but only odd moments represent a novelty since $|Z(t)| = |\zeta(\frac{1}{2}+it)|$. Given $k \in \N$ fixed and $T \geq 1$, consider the $k$-th moment of $Z(t)$, namely
\begin{equation*}
\mathcal{F}_{k}(T) \coloneqq \int_{0}^{T} Z(t)^{k} dt.
\end{equation*}
When $k = 1$, Korolev~\cite{Korolev2008} and Jutila~\cite{Jutila2009,Jutila2011} prove independently the true order~of the first moment
\begin{equation*}
\mathcal{F}_{1}(T) = O(T^{\frac{1}{4}}), \qquad \mathcal{F}_{1}(T) = \Omega_{\pm}(T^{\frac{1}{4}}),
\end{equation*}
Nonetheless, higher odd moments, where one also expects some nontrivial cancellations, still remain unamenable. In particular, the following problem has been unsolved since being posed as a question of study during the conference ``Elementary and Analytic Number Theory'' at the Mathematisches Forschungsinstitut Oberwolfach.
\begin{problem}[Ivi\'{c} 2003]\label{problem-1}
Does there exist a constant $0 < c < 1$ such that
\begin{equation}\label{eq:Ivic-problem}
\mathcal{F}_{3}(T) = O(T^{c})?
\end{equation}
\end{problem}

The approximate functional equation shows~\cite[Equation (7.5)]{Ivic2013} that
\begin{equation}\label{eq:cubic-moment-approximation}
\mathcal{F}_{3}(T) = 2\pi \sqrt{\frac{2}{3}} \sum_{(\frac{T}{2\pi})^{\frac{3}{2}} \leq n \leq (\frac{T}{\pi})^{\frac{3}{2}}} 
d_{3}(n) n^{-\frac{1}{6}} \cos \left(3\pi n^{\frac{2}{3}}+\frac{\pi}{8} \right)+O_{\epsilon}(T^{\frac{3}{4}+\epsilon}),
\end{equation}
where $d_{3}(n)$ is the $3$-fold divisor function generated by $\zeta(s)^{3}$. The problem of obtaining~\eqref{eq:Ivic-problem} for $\frac{3}{4} < c < 1$ boils down to the estimation of the right-hand side of~\eqref{eq:cubic-moment-approximation}. It is known that
\begin{equation}\label{eq:non-linear}
\sum_{n \leq T} d_{3}(n) e(\alpha n^{\frac{2}{3}}) \ll_{\epsilon} T^{\frac{5}{6}+\epsilon}
\end{equation}
uniformly in $\alpha \in \R$. Nevertheless, this corresponds to the exponent $c = 1+\epsilon$ in Problem~\ref{problem-1}, and is in fact weaker than the trivial bound because Cauchy--Schwarz gives
\begin{equation}\label{eq:cubic-moment-best-known}
\mathcal{F}_{3}(T) \ll \left(\int_{0}^{T} \left|\zeta \left(\frac{1}{2}+it \right) \right|^{2} dt \right)^{\frac{1}{2}} 
\left(\int_{0}^{T} \left|\zeta \left(\frac{1}{2}+it \right) \right|^{4} dt \right)^{\frac{1}{2}} \ll T(\log T)^{\frac{5}{2}},
\end{equation}
where we use the classical bounds due to Hardy--Littlewood and Ingham~\cite{Ingham1927}, namely
\begin{equation*}
\int_{0}^{T} \left|\zeta \left(\frac{1}{2}+it \right) \right|^{2} dt \ll T \log T, \qquad 
\int_{0}^{T} \left|\zeta \left(\frac{1}{2}+it \right) \right|^{4} dt \ll T(\log T)^{4}.
\end{equation*}
For the cubic moment of Hardy's function, there seems to exist no better estimate than
~\eqref{eq:cubic-moment-best-known}, and any improvement thereof would necessitate a delicate analysis of the additive twist~\eqref{eq:non-linear}. The difficulty in the estimation of~\eqref{eq:non-linear} lies in the presence of $d_{3}(n)$, which, despite its simple appearance, is quite difficult to handle. In fact, a variant of Vorono\u{\i} summation is essentially involutory and is of no great use in this context. Moreover, writing $n = k \ell m$ transforms~\eqref{eq:non-linear} into a three-dimensional exponential sum over the variables $k, \ell, m \in \N$, but the theory of exponent pairs is not capable of yielding any nontrivial bounds. We also refer the reader to the pioneering papers of Ivi\'{c}~\cite{Ivic2010,Ivic2013,Ivic2017,Ivic2017-3,Ivic2017-2}.

\subsection{$\GL_{3}$ Hardy's Function}
Given $t \in \R$, $\GL_{3}$ Hardy's function is defined by
\begin{equation*}
Z(t, \pi) \coloneqq L \left(\frac{1}{2}+it, \pi \right) \chi \left(\frac{1}{2}+it, \pi \right)^{-\frac{1}{2}}, \qquad 
L(s, \pi) = \chi(s, \pi) L(1-s, \tilde{\pi}).
\end{equation*}
We then define the first moment of $\GL_{3}$ Hardy's function by
\begin{equation*}
\mathcal{F}_{1}(T, \pi) \coloneqq \int_{0}^{T} Z(t, \pi) dt.
\end{equation*}
The key difference from the cubic moment of Hardy's function for the Riemann zeta function is that general $\GL_{3}$ $L$-functions do not factorise into lower degree $L$-functions. This precludes us from using Cauchy--Schwarz as in~\eqref{eq:cubic-moment-best-known}. Instead, if one assumes the second moment bound
\begin{equation*}
\int_{0}^{T} \left|L \left(\frac{1}{2}+it, \pi \right) \right|^{2} dt \ll_{\pi, \epsilon} T^{\frac{3}{2}-\delta+\epsilon}
\end{equation*}
for some $\delta > 0$, then
\begin{equation}\label{eq:delta}
\mathcal{F}_{1}(T, \pi) \ll_{\pi, \epsilon} T^{\frac{5}{4}-\frac{\delta}{2}+\epsilon}.
\end{equation}
Conforming to~\cite[Chapter~7]{Ivic2013} \emph{mutatis mutandis}, one can establish the following reduction.
\begin{theorem}\label{conditional}
Let $\pi \in \mathfrak{F}_{3}$. Assume that
\begin{equation*}
\sum_{n \leq T} \lambda_{\pi}(n) e(\alpha n^{\frac{2}{3}}) \ll_{\pi, \alpha, \epsilon} T^{1-\eta+\epsilon}
\end{equation*}
for some $\eta > 0$. Then we have for any $\epsilon > 0$ that
\begin{equation}\label{eq:eta}
\mathcal{F}_{1}(T, \pi) \ll_{\pi, \epsilon} T^{\frac{5}{4}-\frac{3\eta}{2}+\epsilon}.
\end{equation}
\end{theorem}

Substituting $\eta = \frac{1}{18}$ from Corollary~\ref{corollary} into~\eqref{eq:eta} yields a refined exponent $\frac{7}{6} = 1.16666 \cdots$.
\begin{corollary}\label{cor:first-moment}
Let $\pi \in \mathfrak{F}_{3}$. Then we have for any $\epsilon > 0$ that
\begin{equation*}
\mathcal{F}_{1}(T, \pi) \ll_{\pi, \epsilon} T^{\frac{7}{6}+\epsilon}.
\end{equation*}
\end{corollary}

If $\eta > \frac{1}{6}$, then one can prove the infinitude of zeros of $\GL_{3}$ $L$-functions on the critical line; see the Ph.D. thesis of Rajkumar~\cite{Rajkumar2012}. It is commensurate with Weyl-strength subconvex bounds for $\GL_{3}$ $L$-functions (at least) in the $t$-aspect.

\section{Arithmetic Preliminaries}
This section reviews arithmetic tools that we shall adopt in the proof of Theorem~\ref{main}.

\subsection{Circle Method}\label{circle-method}
There are two oscillatory components in the non-linear additive twist that we address. The idea is to separate these oscillations via the circle method. One seeks for a Fourier expansion that matches the Kronecker symbol $\delta(n = 0)$.
\begin{lemma}[Leung~\cite{Leung2022}]\label{DeltaCor}
Let $n \in \Z$ be such that $|n| \ll N$, $q \in \N$, and let $C > N^{\epsilon}$. Let $U \in C_{c}^{\infty}(\R)$ and $W \in C_{c}^{\infty}([-2, -1] \cup [1, 2])$ be nonnegative even functions such that $U(x) = 1$ for $x \in [-2, 2]$. Then we have that
\begin{equation*}
\delta(n = 0) = \frac{1}{\mathcal{C}} \sum_{c = 1}^{\infty} \frac{1}{cq} 
\sum_{a \tpmod{cq}} e \left(\frac{an}{cq} \right) V_{0} \left(\frac{c}{C}, \frac{n}{cCq} \right),
\end{equation*}
where
\begin{equation*}
\mathcal{C} \coloneqq \sum_{c = 1}^{\infty} W \left(\frac{c}{C} \right) \sim C,
\end{equation*}
and
\begin{equation*}
V_{0}(x, y) \coloneqq W(x) U(x) U(y)-W(y) U(x) U(y)
\end{equation*}
is a smooth function localising to $\delta(|x|, |y| \ll 1)$.
\end{lemma}

In view of the result of Heath-Brown~\cite{HeathBrown1996}, Lemma~\ref{DeltaCor} is the Duke--Friedlander--Iwaniec circle method~\cite{DukeFriedlanderIwaniec1994} with a simpler weight function that restricts $|n| \ll cCQ$. This feature facilitates the forthcoming integral analysis once we make use of the $\mathrm{GL}_{3}$ Vorono\u{\i} summation formula (Lemma~\ref{gl3voronoi}). We direct the reader to~\cite{KanekoLeung2023,Leung2022} for further discussions.

\subsection{Poisson Summation}\label{Poisson-summation}
For $n \in \mathbb{N}$ and an integrable function $w \colon \mathbb{R}^{n} \to \mathbb{C}$, denote its Fourier transform by
\begin{equation*}
\widehat{w}(y) \coloneqq \int_{\R^{n}} w(x) e(-\langle x, y \rangle) dx,
\end{equation*}
where $\langle \cdot, \cdot \rangle$ denotes the standard inner product on $\mathbb{R}^{n}$. Furthermore, if $c \in \N$ and $K \colon \mathbb{Z} \to \mathbb{C}$ is a periodic function of period $c$, then its Fourier transform $\widehat{K}$ is again the periodic function of period $c$, defined on $\Z$ by
\begin{equation*}
\widehat{K}(n) \coloneqq \ \sum_{a \tpmod{c}} K(a) e \left(-\frac{an}{c} \right).
\end{equation*}
There is a minor inconsistency in sign choices, and $\widehat{\widehat{K}}(n) = K(-n)$ for all $n \in \Z$.
\begin{lemma}[{Fouvry--Kowalski--Michel~\cite[Lemma~2.1]{FouvryKowalskiMichel2015}}]\label{Fouvry-Kowalski-Michel}
For any $c \in \mathbb{N}$, any $c$-periodic function $K$, and any even smooth function $V$ compactly supported on $\mathbb{R}$, we have that
\begin{equation*}
\sum_{n = 1}^{\infty} K(n) V(n) = \frac{1}{c} \sum_{n \in \mathbb{Z}} \widehat{K}(n) \widehat{V} \left(\frac{n}{c} \right).
\end{equation*}
\end{lemma}

\subsection{The Gamma Function}\label{the-Gamma-function}
For fixed $\sigma \in \R$, real $|\tau| \geq 3$, and any $M > 0$, we often~make use of Stirling's formula
\begin{equation}\label{eq:Stirling}
\Gamma(\sigma+i\tau) = e^{-\frac{\pi|\tau|}{2}} |\tau|^{\sigma-\frac{1}{2}} \exp \left(i\tau \log \frac{|\tau|}{e} \right) g_{\sigma, M}(\tau)+O_{\sigma, M}(|\tau|^{-M}),
\end{equation}
where
\begin{equation*}
g_{\sigma, M}(\tau) = \sqrt{2\pi} \exp \left(\frac{\pi}{4}(2\sigma-1)i \operatorname{sgn}(\tau) \right)+O_{\sigma, M}(|\tau|^{-1})
\end{equation*}
and
\begin{equation*}
|\tau|^{j} g_{\sigma, M}^{(j)}(\tau) \ll_{j, \sigma, M} 1
\end{equation*}
for all fixed $j \in \N_{0}$.

\section{Automorphic Preliminaries}\label{automorphic-preliminaries}
This section compiles automorphic tools that we shall adopt in the proof of Theorem~\ref{main}; see the book of Goldfeld~\cite{Goldfeld2006} for a thorough account of the theoretical background.

\subsection{Hecke Eigenvalues}\label{automorphic-l-functions}
Denote by $\pi$ a Hecke--Maa{\ss} cusp form of type $(\nu_{1}, \nu_{2})$~for $\mathrm{SL}_{3}(\Z)$ with the Fourier--Whittaker coefficients $\lambda_{\pi}(m, n)$ normalised so that $\lambda_{\pi}(1, 1) = 1$. As usual, one abbreviates $\lambda_{\pi}(1, n) \coloneqq \lambda_{\pi}(n)$. The corresponding Langlands parameters are given by
\begin{equation*}
\mu_{1} = -\nu_{1}-2\nu_{2}+1, \qquad \mu_{2} = -\nu_{1}+\nu_{2}, \qquad \mu_{3} = 2\nu_{1}+\nu_{2}-1.
\end{equation*}
The Ramanujan conjecture then states that $\Re(\mu_{i}) = 0$. The work of Jacquet--Shalika~\cite{JacquetShalika1990-2} shows $|\Re(\mu_{i})| < \frac{1}{2}$. The Fourier--Whittaker coefficients obey the Hecke relation
\begin{equation}\label{Hecke}
\lambda_\pi(n_{1}, 1) \lambda_{\pi}(1, n_{2}) = \sum_{d \mid (n_{1}, n_{2})} \lambda_{\pi} \left(\frac{n_1}{d}, \frac{n_2}{d} \right).
\end{equation}
The Rankin--Selberg bound is known in the following form.
\begin{lemma}\label{Ram bound}
We have for any~$\epsilon > 0$ that
\begin{equation*}
\underset{n_{1}^{2} n_{2} \leq N}{\sum \sum} |\lambda_{\pi}(n_{2}, n_{1})|^{2}
\ll_{\pi, \epsilon} N^{1+\epsilon}.
\end{equation*}
\end{lemma}

As a corollary, we deduce the following estimate.

\begin{corollary}\label{GL3RPavg}
We have for any~$\epsilon > 0$ that
    \begin{equation*}
        \sum_{n_1\leq N_1}\sum_{n_2\leq N_2}|\lambda_\pi(n_1,n_2)|^2\ll (N_1N_2)^{1+\epsilon}.
    \end{equation*}
\end{corollary}
\begin{proof}
    Applying the Hecke relation~\eqref{Hecke} in tandem with M\"{o}bius inversion yields
    \begin{align*}
        \sum_{n_1\leq N_1}\sum_{n_2\leq N_2}|\lambda_\pi(n_1,n_2)|^2=&\sum_{n_1\leq N_1}\sum_{n_2\leq N_2}\left|\sum_{d \mid(n_1,n_2)}\mu(d)\lambda_\pi \left(\frac{n_1}{d}, 1 \right) \lambda_\pi \left(1,\frac{n_2}{d} \right)\right|^2.
    \end{align*}
    By the Cauchy--Schwarz inequality to take out the $d$-sum, the right-hand side is bounded by
    \begin{equation*}
        X^\epsilon\sum_{n_1\leq N_1}\sum_{n_2\leq N_2}\sum_{d \mid(n_1,n_2)} \left|\lambda_\pi \left(\frac{n_1}{d},1 \right)\lambda_\pi \left(1,\frac{n_2}{d} \right) \right|^2\\
        = X^\epsilon \sum_d\sum_{n_1\leq \frac{N_1}{d}}\sum_{n_2\leq \frac{N_2}{d}} |\lambda_\pi(n_1,1)\lambda_\pi(1,n_2)|^2.
    \end{equation*}
    Corollary~\ref{GL3RPavg} now follows from Lemma~\ref{Ram bound}.
\end{proof}

We record the best known bound for individual Fourier--Whittaker coefficients $\lambda_{\pi}(n_{1}, n_{2})$.
\begin{lemma}[Kim--Sarnak~\cite{Kim2003}]\label{KimSarnak}
We have for any~$\epsilon > 0$ that
\begin{equation*}
|\lambda_{\pi}(n_{1}, n_{2})| \ll_{\pi, \epsilon} (n_{1} n_{2})^{\frac{5}{14}+\epsilon}.
\end{equation*}
\end{lemma}

\subsection{$\GL_{3}$ Vorono\u{\i} Summation}\label{subsec:Voronoi}
In conjunction with Poisson summation in Section~\ref{Poisson-summation},~the proof of Theorem~\ref{main} requires the Vorono\u{\i} summation formula attached to Hecke--Maa{\ss} cusp forms on $\mathrm{Z}(\mathbb{R}) \mathrm{SL}_{3}(\mathbb{Z}) \backslash \mathrm{GL}_{3}(\mathbb{R})/\mathrm{O}(3)$. Let $h$ be a compactly supported smooth function on~$\R_{+}$ and denote the corresponding Mellin transform by
\begin{equation*}
\tilde{h}(s) \coloneqq \int_{0}^{\infty} h(x) x^{s-1} dx.
\end{equation*}
Given $\sigma > -1+\max\{-\Re(\alpha_{1}), -\Re(\alpha_{2}), -\Re(\alpha_{3}) \}$ and $a \in \{0, 1 \}$, define
\begin{equation*}
\gamma_{a}(s) \coloneqq \frac{\pi^{-3s-\frac{3}{2}}}{2} \prod_{i = 1}^{3} \frac{\Gamma(\frac{1+s+\alpha_{i}+a}{2})}{\Gamma(\frac{-s-\alpha_{i}+a}{2})}.
\end{equation*}
Let $\gamma_{\pm}(s) \coloneqq \gamma_{0}(s) \mp i \gamma_{1}(s)$, and define
\begin{equation}\label{H_pm}
\mathcal{H}_{\pm}(x) \coloneqq \frac{1}{2\pi i} \int_{(\sigma)} x^{-s} \gamma_{\pm}(s) \tilde{h}(-s) ds.
\end{equation}
We now formulate a version of the $\GL_{3}$ Vorono\u{\i} summation formula; cf.~\cite{Blomer2012,Li2011,MillerSchmid2006}.
\begin{lemma}[$\mathrm{GL}_{3}$ Vorono\u{\i} summation]\label{gl3voronoi}
Let $a, \overline{a} \in \mathbb{Z}$, $q \in \mathbb{N}$ with $(a, q) = 1$,~$a \overline{a} \equiv 1 \tpmod{q}$. Let $h$ be a compactly supported smooth function on $\R_{+}$. Then we have that
\begin{equation*}
\sum_{n = 1}^{\infty} \lambda_{\pi}(n) e \left(\frac{an}{q} \right) h(n)
 = q \sum_{\pm} \sum_{n_{0} \mid q} \sum_{n = 1}^{\infty} \frac{\lambda_{\pi}(n, n_{0})}{nn_{0}} S \left(\overline{a}, \pm n; \frac{q}{n_{0}} \right) \mathcal{H}_{\pm} \left(\frac{nn_{0}^{2}}{q^{3}} \right).
\end{equation*}
\end{lemma}

The behaviour of the integral transform $\mathcal{H}_{\pm}$ on the dual side is determined in the literature.
\begin{lemma}[{Blomer~\cite[Lemma 6]{Blomer2012}}]\label{voronoi}
Let $h$ be a compactly supported smooth function on $[a, b] \subset (0, \infty)$, and let $\mathcal{H}_{\pm}$ be defined by~\eqref{H_pm}. Then there exist constants $\gamma_{\ell}$ depending~only on the Langlands parameters $(\alpha_{1}, \alpha_{2}, \alpha_{3})$ such that
\begin{equation*}
\mathcal{H}_{\pm}(x) = x \int_{0}^{\infty} h(y) \sum_{\ell = 1}^{L} \frac{\gamma_{\ell}}{(xy)^{\frac{\ell}{3}}}e(\pm 3(xy)^{\frac{1}{3}}) dy + \mathcal{R}(x),
\end{equation*}
where the error term $\mathcal{R}(x)$ satisfies
\begin{equation*}
x^{k} \frac{d^{k}}{dx^{k}} \mathcal{R}(x) \ll_{} \norm{h}_\infty^{1+\frac{k-L}{3}}
\end{equation*}
with an implicit constant depending at most on $a, b, \alpha_{1}, \alpha_{2}, \alpha_{3}, k$, and $L$.
\end{lemma}

\section{Proof of Theorem~\ref{main}}\label{proof-of-the-main-theorem}
In this section, we embark on the proof of Theorem~\ref{main}. Recall that our goal is to estimate
\begin{equation*}
    \mathcal{A}_{\pi}(T, \alpha, \beta) \coloneqq \sum_{n \leq T} \lambda_{\pi}(n) e(\alpha n^{\beta}) \ll_{\pi, \alpha, \beta, \epsilon} T^{\frac{9+10\beta}{17}+\epsilon}+T^{\frac{19}{14}-\beta+\epsilon}.
\end{equation*}
We will consider $\pi$, $\alpha$, $\beta$ fixed, and employ the usual convention that $\epsilon$ is an arbitrarily small positive quantity, which may change from line to line. We will then drop the dependence on $\pi$, $\alpha$, $\beta$, $\epsilon$, but each inequality is allowed to possess an implicit constant depending on them. Without loss of generality, we shall assume that $\alpha>0$.

\subsection{Smoothing}
A dyadic subdivision yields
\begin{equation*}
\mathcal{A}_{\pi}(T, \alpha, \beta) \ll T^{\epsilon} \sup_{1 \leq N \leq T} \sum_{\frac{N}{2} < n \leq N} \lambda_{\pi}(n) e(\alpha n^{\beta}).
\end{equation*}
Let $Y \geq 1$ be a parameter at our disposal. Approximating the indicator function $\mathbf{1}_{(\frac{N}{2},N]}$ by~a $Y$-inert function $W(\frac{\cdot}{N})$ that takes $1$ on $[\frac{1}{2}+\frac{1}{Y}, 1-\frac{1}{Y}]$ and $0$ outside $[\frac{1}{2}, 1]$, and bounding the error term with Lemma~\ref{KimSarnak}, we deduce
\begin{equation}\label{DyadicSmoothing}
\mathcal{A}_\pi(T, \alpha, \beta) \ll T^{\epsilon} \sup_{1 \leq N \leq T} \left(|\mathcal{S}(N)| + \frac{N^{\frac{19}{14}+\epsilon}}{Y}\right),
\end{equation}
where
\begin{equation*}
\mathcal{S}(N) \coloneqq \sum_{n} \lambda_{\pi}(n) e(\alpha n^{\beta}) W \left(\frac{n}{N} \right).
\end{equation*}

\subsection{Applying the $\delta$-Symbol}
We now use Lemma~\ref{DeltaCor} to separate the oscillation of $\lambda_\pi(n)$ from $e(\alpha n^{\beta}) W(\frac{n}{N})$. Let $C \geq 1$ be a parameter at our disposal, and fix a smooth function $U$ that takes $1$ on $[1, 2]$ and $0$ outside $[\frac{1}{2}, \frac{5}{2}]$. Then
\begin{align*}
\mathcal{S}(N) &= \sum_{m} \lambda_{\pi}(m) U \left(\frac{m}{N} \right) \sum_{n} e(\alpha n^{\beta}) W \left(\frac{n}{N} \right) \delta(n = m)\\
& = \frac{1}{\mathcal{C}} \sum_{c} \frac{1}{c} \sum_{m} \lambda_{\pi}(m) U \left(\frac{m}{N} \right) \sum_{n} e(\alpha n^{\beta}) W \left(\frac{n}{N} \right) \sum_{a \tpmod{c}} e \left(\frac{a(m-n)}{c} \right) V_{0} \left(\frac{c}{C}, \frac{m-n}{cC} \right)
\end{align*}
for some $\mathcal{C} \sim C$ and a fixed smooth function $V_{0}$ satisfying $V_{0}(x, y) \ll \delta(|x|, |y| \ll 1)$. Pulling out the divisor $(a, c)$ implies
\begin{multline*}
    \mathcal{S}(N) \ll \frac{1}{C} \sum_{b} \sum_{c} \frac{1}{bc} \sumast_{a \tpmod{c}} \sum_{m} \lambda_{\pi}(m) e \left(\frac{a m}{c} \right) U \left(\frac{m}{N} \right)\\
    \times \sum_{n} e \left(\alpha n^{\beta}-\frac{an}{c} \right) W \left(\frac{n}{N} \right) V_{0} \left(\frac{bc}{C}, \frac{m-n}{bcC} \right).
\end{multline*}

\subsection{Vorono\u{\i} Summation in \texorpdfstring{$m$}{m}}\label{Voronoi-summation-in-m}
Applying $\mathrm{GL}_3$ Vorono\u{\i} summation (Lemma~\ref{gl3voronoi} together with Lemma~\ref{voronoi}) to the $m$-sum, we deduce
\begin{multline*}
    \sum_{m} \lambda_{\pi}(m) e \left(\frac{a m}{c} \right)U\left(\frac{m}{N} \right)V_{0} \left(\frac{bc}{C}, \frac{m-n}{bcC} \right)\\
    = \frac{N^{\frac{2}{3}}}{c} \sum_\pm\sum_{m_{0} \mid c} \sum_{m} \lambda_\pi(m,m_{0})\left(\frac{m_{0}}{m} \right)^{\frac{1}{3}}S\left(\overline{a}, \pm m;\frac{c}{m_{0}} \right) \mathcal{J}_{1}(m_{0}^{2}m,n,c),
\end{multline*}
where\footnote{We suppress less important variables from the notation, and it applies to the ensuing integral transforms.}
\begin{equation}\label{eq:J1}
    \mathcal{J}_{1}(m_{0}^{2}m,n,c) \coloneqq \int_{0}^\infty \Phi\left(y, \frac{m_{0}^{2}m}{c^3} \right)V_{0} \left(\frac{bc}{C}, \frac{Ny-n}{bcC} \right)e \left(\pm\frac{3(m_{0}^{2}mNy)^{\frac{1}{3}}}{c} \right)dy
\end{equation}
for some flat function $\Phi$ supported on $[\frac{1}{2}, \frac{5}{2}] \times \R$. This leads to the expression
\begin{multline*}
    \mathcal{S}(N)\ll \frac{N^{\frac{2}{3}}}{C} \sum_\pm\sum_b\sum_c \frac{1}{bc^{2}} \sumast_{a\tpmod{c}} \sum_{m_{0} \mid c} \sum_{m} \lambda_{\pi}(m,m_{0})\left(\frac{m_{0}}{m} \right)^{\frac{1}{3}} S\left(\overline{a}, \pm m;\frac{c}{m_{0}} \right)\\
    \times \sum_{n} e \left(\alpha n^{\beta}-\frac{an}{c} \right) W \left(\frac{n}{N} \right) \mathcal{J}_{1}(m_{0}^{2}m,n,c).
\end{multline*}

\subsection{Poisson Summation in \texorpdfstring{$n$}{n}}\label{Poisson-summation-in-n}
Applying Poisson summation (Lemma~\ref{Fouvry-Kowalski-Michel}) to the $n$-sum, we deduce
\begin{equation*}
    \sum_{n} e \left(\alpha n^{\beta}-\frac{a n}{c} \right) W \left(\frac{n}{N} \right) \mathcal{J}_{1}(m_{0}^{2}m,n,c)
    =\frac{N}{c} \sum_{n\in\Z} \sum_{\gamma\tpmod{c}}e \left(\frac{(n-a)\gamma}{c} \right) \mathcal{J}_{2}(m_{0}^{2}m,n,c),
\end{equation*}
where
\begin{align}\label{eq:J2}
    \begin{split}
    \mathcal{J}_{2}(m_{0}^{2}m,n,c) &\coloneqq \int_{0}^\infty W(x) \mathcal{J}_{1}(m_{0}^{2}m,Nx,c)e \left(\alpha(Nx)^\beta-\frac{nNx}{c} \right)dx\\
    & = \int_{0}^\infty\int_{0}^\infty W(x)\Phi\left(y, \frac{m_{0}^{2}m}{c^3} \right)V_{0} \left(\frac{bc}{C}, \frac{N(y-x)}{bcC} \right)\\
    & \quad \times e \left(\pm\frac{3(m_{0}^{2}mNy)^{\frac{1}{3}}}{c}+\alpha(Nx)^\beta-\frac{nNx}{c} \right)dxdy.
    \end{split}
\end{align}
The sum over $\gamma$ is equal to $c \cdot \delta(n \equiv a \tpmod{c})$. Making a change of variables $c\mapsto cm_{0}$~gives
\begin{equation*}
    \mathcal{S}(N)\ll \frac{N^{\frac{5}{3}}}{C} \sum_\pm\sum_b\sum_c \frac{1}{bc^{2}} \sum_{m_{0}} \sum_{m} \frac{\lambda_{\pi}(m,m_{0})}{m_{0}^{\frac{5}{3}}m^{\frac{1}{3}}} \sum_{\substack{n\in\Z\\(n,cm_0)=1}} S(\pm m, \overline{n}; c) \mathcal{J}_{2}(m_{0}^{2}m,n,cm_{0}).
\end{equation*}
If we choose $Y\leq N^\beta T^{-\epsilon}$ and $C\leq \sqrt{N}$, then repeated integration by parts in $x$ and $y$ in~\eqref{eq:J2} ensures an arbitrary saving unless
\begin{equation*}
    m_{0}^{2}m \ll \frac{N^{2}T^\epsilon}{b^3C^3}, \qquad n>0, \qquad \frac{cm_0T^{-\epsilon}}{N^{1-\beta}}\ll n\ll \frac{cm_0T^\epsilon}{N^{1-\beta}}.
\end{equation*}
Therefore, we arrive at
\begin{multline*}
    \mathcal{S}(N)\ll \frac{N^{\frac{5}{3}}}{C} \sum_\pm\sum_b\sum_c \frac{1}{bc^{2}} \mathop{\sum\sum}_{m_{0}^{2}m\ll\frac{N^{2}T^\epsilon}{b^3C^3}} \frac{\lambda_{\pi}(m,m_{0})}{m_{0}^{\frac{5}{3}}m^{\frac{1}{3}}} \\
    \times \sum_{\substack{\frac{cm_0T^{-\epsilon}}{N^{1-\beta}}\ll n\ll \frac{cm_0T^\epsilon}{N^{1-\beta}}\\(n,cm_0)=1}} S(\pm m, \overline{n}; c) \mathcal{J}_{2}(m_{0}^{2}m,n,cm_{0})+T^{-2023}.
\end{multline*}

\subsection{Rewriting the Integral Transform}
In anticipation of the Cauchy--Schwarz step,~it is convenient to rewrite slightly the integral $\mathcal{J}_2$ and extract some of its oscillations. Making~a change of variables
\begin{equation*}
    Ny \mapsto Nx+bcm_0Cu,
\end{equation*}
we obtain
\begin{multline*}
    \mathcal{J}_{2}(m_{0}^{2}m,n,cm_0) = \frac{bcm_0C}{N}\int_\R\int_{0}^\infty W(x)\Phi\left(x+\frac{bcm_0Cu}{N}, \frac{m}{c^3m_0} \right)V_{0} \left(\frac{bcm_0}{C}, u \right)\\
    \times e \left(\pm\frac{3(m(Nx+bcm_0Cu))^{\frac{1}{3}}}{cm_0^{\frac{1}{3}}}+\alpha(Nx)^\beta-\frac{nNx}{cm_0} \right)dxdu.
\end{multline*}

\begin{remark}
    This is one of the strengths of the circle method that we choose. The bump function $V_{0}$ being fixed allows one to extract the oscillations via a simple change of variables.
\end{remark}

Note that for any $x\in\R$ and $-1<\sigma<0$,
\begin{equation}\label{eq:Taylor}
    e^{ix}-1 = \int_{(\sigma)} \Gamma(s) e \left(\frac{s}{4} \right) x^{-s}\frac{ds}{2\pi i}.
\end{equation}
This is immediate from shifting the contour to the left and then comparing it with the Taylor expansion of $e^{ix}-1$. The identity~\eqref{eq:Taylor} gives
\begin{align*}
    \mathcal{J}_2(m_0^2m,n,cm_0)& = \frac{bcm_0C}{N}\int_\R\int_{0}^\infty W(x)\Phi\left(x+\frac{bcm_0Cu}{N}, \frac{m}{c^3m_0} \right)\\
    & \quad \times V_{0} \left(\frac{bcm_0}{C}, u \right)e \left(\alpha(Nx)^\beta-\frac{nNx}{cm_0} \right)\\
    & \quad \times \left(1+\int_{(-\frac{1}{2}-\epsilon)}\Gamma(s)e\left(\frac{s}{4}\right)\left(\pm\frac{3(m(Nx+bcm_0Cu))^{\frac{1}{3}}}{cm_0^{\frac{1}{3}}}\right)^{-s}\frac{ds}{2\pi i}\right)dxdu.
\end{align*}
It follows from Stirling's formula~\eqref{eq:Stirling} that the $s$-integral converges absolutely. Furthermore, repeated integration by parts in the $s$-integral ensures an arbitrary saving unless
\begin{equation*}
    |\Im(s)|\ll \left(\frac{mN}{c^{3} m_0} \right)^{\frac{1}{3}} T^\epsilon.
\end{equation*}

\subsection{Cauchy--Schwarz in $m$}\label{Cauchy-Schwarz}
Applying dyadic subdivisions on $m_0^2m\sim M$ and $cm_0\sim C_0$, and the Cauchy--Schwarz inequality with Corollary~\ref{GL3RPavg} to take out the $(b,m_0,m)$-sums and $(u,s)$-integrals with the truncation $|\Im(s)|\ll \frac{(MN)^{\frac{1}{3}}T^\epsilon}{C_0}$, we derive
\begin{equation}\label{RNSplit}
    \mathcal{S}(N)\ll \sup_{\pm}\sup_{\eta \in \{0,1 \}}\sum_b \sup_{m_0\sim M_0}\sup_{|u|\ll1}\sup_{C_0\ll \frac{C}{b}}\sup_{M\ll \frac{N^2T^\epsilon}{b^3C^3}}\int_{|t|\ll \frac{(MN)^{\frac{1}{3}} T^\epsilon}{C_0}}\frac{M^{\frac{1}{6}}N^{\frac{2}{3}}T^\epsilon}{\sqrt{1+|t|^{2+\epsilon}}} \sqrt{\mathcal{R}(N)} dt+T^{-2023},
\end{equation}
where 
\begin{equation*}
    \mathcal{R}(N) \coloneqq \sum_{M\leq m_{0}^{2}m<2M}\left|\sum_{\frac{C_0}{m_0}\leq c<\frac{2C_0}{m_0}} \frac{1}{c} \sum_{\substack{\frac{C_0T^{-\epsilon}}{N^{1-\beta}}\ll n\ll \frac{C_0T^\epsilon}{N^{1-\beta}}\\(n,cm_0)=1}} S(\pm m, \overline{n}; c) \mathcal{J}_{3}^\eta(m,n,c)\right|^2
\end{equation*}
with
\begin{multline}\label{eq:J3}
    \mathcal{J}_{3}^\eta( m, n, c) \coloneqq V_{0} \left(\frac{bcm_0}{C}, u \right) \int_{0}^\infty W(x)\Phi\left(x+\frac{bcm_0Cu}{N}, \frac{m}{c^3m_0} \right)\\
    \times e \left(\alpha(Nx)^\beta-\frac{nNx}{cm_0} \right)\left(\frac{(MN)^{\frac{1}{6}}}{\sqrt{m_0}}\left(\frac{x+\frac{bcm_0Cu}{N}}{c^3}\right)^{\frac{1}{6}+\frac{\epsilon}{3}-\frac{it}{3}} \right)^{\eta}dx.
\end{multline}
Before we open the square, we lengthen the $m$-sum in anticipation of Poisson summation in the ensuing analysis. Let $L\geq 1$ be a parameter at our disposal, and fix a smooth function~$U_0$ supported on $[-3,3]$ such that $U_0(x)=1$ for $x \in[-2, 2]$. Then
\begin{equation}\label{RNafterlengthening}
    \mathcal{R}(N)\leq \sum_mU_0\left(\frac{m_0^2m}{LM}\right)\left|\sum_{\frac{C_0}{m_0}\leq c<\frac{2C_0}{m_0}} \frac{1}{c} \sum_{\substack{\frac{C_0T^{-\epsilon}}{N^{1-\beta}}\ll n\ll \frac{C_0T^\epsilon}{N^{1-\beta}}\\(n,cm_0)=1}} S(\pm m, \overline{n}; c) \mathcal{J}_{3}^\eta(m,n,c)\right|^2.
\end{equation}
This step counterweights zero frequency and nonzero frequencies in Poisson summation, and the optimisation of the parameter $L$ thus governs the quality of the final bound.

\subsection{Poisson Summation in \texorpdfstring{$m$}{m}}\label{Poisson-summation-in-m}

Opening the square, the $m$-sum boils down to
\begin{equation*}
    \sum_m U_0\left(\frac{m_0^2m}{LM}\right) S(\pm m, \overline{n_1}; c_1) S(\pm m, \overline{n_2}; c_2) \mathcal{J}_{3}^\eta(m,n_1,c_1)\overline{\mathcal{J}_{3}^\eta(m,n_2,c_2)}.
\end{equation*}
Poisson summation (Lemma~\ref{Fouvry-Kowalski-Michel}) then yields
\begin{equation*}
    \frac{LM}{c_1c_2}\sum_m\mathcal{C}^\pm(m,n_1,n_2;c_1,c_2)\mathcal{J}_4^\eta(m,n_1,n_2,c_1,c_2),
\end{equation*}
where
\begin{equation*}
    \mathcal{C}^\pm(m,n_1,n_2;c_1,c_2)\coloneqq \sum_{\gamma\tpmod{c_2c_2}}S(\pm \gamma, \overline{n_1}; c_1) S(\pm \gamma, \overline{n_2}; c_2)e\left(\frac{m\gamma}{c_1c_2}\right)
\end{equation*}
and
\begin{equation}\label{eq:J4}
    \mathcal{J}_4^\eta(m,n_1,n_2,c_1,c_2)\coloneqq \int_\R U_0(y) \mathcal{J}_{3}^\eta(LMy,n_1,c_1)\overline{\mathcal{J}_{3}^\eta(LMy,n_2,c_2)}e\left(-\frac{mLMy}{c_1c_2}\right)dy.
\end{equation}
Repeated integration by parts in the $y$-integral ensures an arbitrary saving unless
\begin{equation*}
    |m|\ll \frac{C_0^2T^\epsilon}{m_0^2LM}.
\end{equation*}
Inserting the above manipulations back into~\eqref{RNafterlengthening}, we obtain
\begin{multline}\label{RNAfterPoisson}
    \mathcal{R}(N)\ll LM\mathop{\sum\sum}_{\frac{C_0}{m_0}\leq c_1,c_2<\frac{2C_0}{m_0}} \frac{1}{(c_1c_2)^2} \mathop{\sum\sum}_{\substack{\frac{C_0T^{-\epsilon}}{N^{1-\beta}}\ll n_1,n_2\ll \frac{C_0T^\epsilon}{N^{1-\beta}}\\(n_1,c_1m_0)=(n_2,c_2m_0)=1}}\\
    \times \sum_{|m|\ll \frac{C_0^2T^\epsilon}{m_0^2LM}}\mathcal{C}^\pm(m,n_1,n_2;c_1,c_2)\mathcal{J}_4^\eta(m,n_1,n_2,c_1,c_2) + T^{-2023}.
\end{multline}
For the character sum $\mathcal{C}^\pm$, summing over $\gamma\tpmod{c_1c_2}$ yields
\begin{equation*}   \mathcal{C}^\pm(m,n_1,n_2;c_1,c_2)=c_1c_2\mathop{\sumast_{\alpha_1\tpmod{c_1}}\sumast_{\alpha_2\tpmod{c_2}}}_{\alpha_1c_2+\alpha_2c_1\equiv \mp m\tpmod{c_1c_2}}e\left(\frac{\overline{\alpha_1n_1}}{c_1}+\frac{\overline{\alpha_2n_2}}{c_2}\right).
\end{equation*}
Write $c_0=(c_1,c_2)$, $c_{1,0}=(\frac{c_1}{c_0},c_0^\infty)$, $c_{2,0}=(\frac{c_2}{c_0},c_0^\infty)$, $c_1=c_0c_{1,0}c_1'$, $c_2=c_0c_{2,0}c_2'$, and $(c_1',c_2')=(c_1'c_2',c_0)=1$. The congruence condition $\alpha_{1} c_{2}+\alpha_{2} c_{1} \equiv \mp m \tpmod{c_{1} c_{2}}$ then decomposes as
\begin{gather*}
    c_0 \mid m, \qquad \alpha_1c_{2,0}c_2'+\alpha_2c_{1,0}c_1'\equiv \mp \frac{m}{c_0} \tpmod{c_0c_{1,0}c_{2,0}},\\
    \alpha_1\equiv \mp m\overline{c_2}\tpmod{c_1'}, \qquad \alpha_2\equiv \mp m\overline{c_1}\tpmod{c_2'}.
\end{gather*}
By the Chinese Remainder Theorem, we observe
\begin{multline*}
    \mathcal{C}^\pm(m,n_1,n_2;c_1,c_2)=c_1c_2e\left(\mp\frac{c_2\overline{c_0c_{1,0}mn_1}}{c_1'}\mp\frac{c_1\overline{c_0c_{2,0}mn_2}}{c_2'}\right)\\
    \times \delta(c_0 \mid m)\mathop{\sumast_{\alpha_1\tpmod{c_0c_{1,0}}}\sumast_{\alpha_2\tpmod{c_0c_{2,0}}}}_{ \alpha_1c_{2,0}c_2'+\alpha_2c_{1,0}c_1'\equiv \mp \frac{m}{c_0} \tpmod{c_0c_{1,0}c_{2,0}}}e\left(\frac{\overline{\alpha_1c_1'n_1}}{c_0c_{1,0}}+\frac{\overline{\alpha_2c_2'n_2}}{c_0c_{2,0}}\right).
\end{multline*}
In particular, if $m=0$, then
\begin{align*}
    \mathcal{C}^\pm(0,n_1,n_2;c_1,c_2) &= c_1^2\delta(c_1=c_2)\sumast_{\alpha\tpmod{c_1}}e\left(\frac{\alpha(\overline{n_1}-\overline{n_2})}{c_1}\right)\\
    &= c_1^2\sum_{ac'=c_1}\mu(a)c'\delta(c_1=c_2, n_1\equiv n_2\tpmod{c'}).
\end{align*}
Hence, the contribution of $m=0$ to~\eqref{RNAfterPoisson} is given by
\begin{equation*}
    LM\sum_{\frac{C_0}{m_0}\leq c<\frac{2C_0}{m_0}} \frac{1}{c^2}\sum_{ac'=c}\mu(a)c' \mathop{\sum\sum}_{\substack{\frac{C_0T^{-\epsilon}}{N^{1-\beta}}\ll n_1,n_2\ll \frac{C_0T^\epsilon}{N^{1-\beta}}\\n_1\equiv n_2\tpmod{c'}\\(n_1n_2,cm_0)=1}}\mathcal{J}_4^\eta(0,n_1,n_2,c_1,c_2).
\end{equation*}
Applying the second derivative test on the integral transform $\mathcal{J}_{3}^{\eta}$ in~\eqref{eq:J3} yields
\begin{equation*}
    \mathcal{J}_4^\eta(0,n_1,n_2,c_1,c_2)=\int_\R U_0(y) \mathcal{J}_{3}^\eta(LMy,n_1,c_1)\overline{\mathcal{J}_{3}^\eta(LMy,n_2,c_2)}dy\ll \frac{(MN)^{\frac{1}{3}} T^{\epsilon}}{C_0N^{\beta}},
\end{equation*}
where we use $Y \leq N^\beta T^{-\epsilon}$. As a result, the contribution of $m=0$ to~\eqref{RNAfterPoisson} is bounded by \begin{equation*}
    LM\sum_{\frac{C_0}{m_0}\leq c<\frac{2C_0}{m_0}} \frac{1}{c^2}\sum_{ac'=c}\mu(a)c' \mathop{\sum\sum}_{\substack{\frac{C_0T^{-\epsilon}}{N^{1-\beta}}\ll n_1,n_2\ll \frac{C_0T^\epsilon}{N^{1-\beta}}\\n_1\equiv n_2\tpmod{c'}}}\frac{(MN)^{\frac{1}{3}} T^{\epsilon}}{C_0N^{\beta}} \ll \frac{LM^{\frac{4}{3}} T^\epsilon}{N^{\frac{2}{3}}}.
\end{equation*}
It now transpires that
\begin{align*}
    \begin{split}
    \mathcal{R}(N) &\ll \frac{LM^{\frac{4}{3}} T^\epsilon}{N^{\frac{2}{3}}}+LM\sum_{0\neq |m|\ll \frac{C_0^2T^\epsilon}{m_0^2LM}}\sum_{c_0 \mid m}\mathop{\sum\sum}_{\substack{c_{1,0},c_{2,0} \mid c_0^\infty\\ (c_{1,0},c_{2,0})=1}}\mathop{\sum\sum}_{\substack{\frac{C_0}{c_0m_0}\leq c_{1,0}c_1',c_{2,0}c_2'<\frac{2C_0}{c_0m_0}\\(c_1',c_2')=(c_1'c_2',m)=1}} \frac{1}{c_0^2c_{1,0}c_{2,0}c_1'c_2'}\\
    & \quad \times \mathop{\sum\sum}_{\substack{\frac{C_0T^{-\epsilon}}{N^{1-\beta}}\ll n_1,n_2\ll \frac{C_0T^\epsilon}{N^{1-\beta}}\\(n_1,c_0c_1'm_0)=(n_2,c_0c_2'm_0)=1}} \mathop{\sumast_{\alpha_1\tpmod{c_0c_{1,0}}}\sumast_{\alpha_2\tpmod{c_0c_{2,0}}}}_{ \alpha_1c_{2,0}c_2'+\alpha_2c_{1,0}c_1'\equiv \mp \frac{m}{c_0}\tpmod{c_0c_{1,0}c_{2,0}}}e\left(\frac{\overline{\alpha_1c_1'n_1}}{c_0c_{1,0}}+\frac{\overline{\alpha_2c_2'n_2}}{c_0c_{2,0}}\right)\\
    & \quad \times e\left(\mp\frac{c_{2,0}c_2'\overline{c_{1,0}mn_1}}{c_1'}\mp\frac{c_{1,0}c_1'\overline{c_{2,0}mn_2}}{c_2'}\right)\mathcal{J}_4^\eta(m,n_1,n_2,c_0c_{1,0}c_1',c_0c_{2,0}c_2').
    \end{split}
\end{align*}

\subsection{Additive Reciprocity}\label{additive-reciprocity}
Additive reciprocity implies
\begin{align*}
    \mathcal{R}(N) &\ll \frac{LM^{\frac{4}{3}} T^\epsilon}{N^{\frac{2}{3}}}+LM\sum_{0\neq |m|\ll \frac{C_0^2T^\epsilon}{m_0^2LM}}\sum_{c_0 \mid m}\mathop{\sum\sum}_{\substack{c_{1,0},c_{2,0} \mid c_0^\infty\\ (c_{1,0},c_{2,0})=1}}\mathop{\sum\sum}_{\substack{\frac{C_0}{c_0m_0}\leq c_{1,0}c_1',c_{2,0}c_2'<\frac{2C_0}{c_0m_0}\\(c_1',c_2')=(c_1'c_2',m)=1}} \frac{1}{c_0^2c_{1,0}c_{2,0}c_1'c_2'}\\
    & \quad \times \mathop{\sum\sum}_{\substack{\frac{C_0T^{-\epsilon}}{N^{1-\beta}}\ll n_1,n_2\ll \frac{C_0T^\epsilon}{N^{1-\beta}}\\(n_1,c_0c_1'm_0)=(n_2,c_0c_2'm_0)=1}} \mathop{\sumast_{\alpha_1\tpmod{c_0c_{1,0}}}\sumast_{\alpha_2\tpmod{c_0c_{2,0}}}}_{ \alpha_1c_{2,0}c_2'+\alpha_2c_{1,0}c_1'\equiv \mp \frac{m}{c_0}\tpmod{c_0c_{1,0}c_{2,0}}}e\left(\frac{\overline{\alpha_1c_1'n_1}}{c_0c_{1,0}}+\frac{\overline{\alpha_2c_2'n_2}}{c_0c_{2,0}}\right)\\
    & \quad \times e\left(\pm\frac{c_{2,0}c_2'\overline{c_1'}}{c_{1,0}mn_1}\pm\frac{c_{1,0}c_1'\overline{c_2'}}{c_{2,0}mn_2}\right)\mathcal{J}_5^\eta(m,n_1,n_2,c_0,c_{1,0},c_{2,0},c_1',c_2'),
\end{align*}
where
\begin{multline*}
    \mathcal{J}_5^\eta(m,n_1,n_2,c_0,c_{1,0},c_{2,0},c_1',c_2')\\
    \coloneqq e\left(\mp\frac{c_{2,0}c_2'}{c_{1,0}c_1'mn_1}\mp\frac{c_{1,0}c_1'}{c_{2,0}c_2'mn_2}\right)\mathcal{J}_4^\eta(m,n_1,n_2,c_0c_{1,0}c_1',c_0c_{2,0}c_2').
\end{multline*}
Tracing back to the definitions~\eqref{eq:J1},~\eqref{eq:J2},~\eqref{eq:J3}, and~\eqref{eq:J4} shows
\begin{multline*}
    \mathcal{J}_5^\eta(m,n_1,n_2,c_0,c_{1,0},c_{2,0},c_1',c_2')\\
    =T^{2\epsilon}\left(\frac{(MN)^{\frac{1}{3}}}{C_0}\right)^\eta \int_\R \int_{0}^\infty\int_{0}^\infty  W(x_1)W(x_2) U_0(y)
    \Phi_{m,n_1,n_2}^{c_0,c_{1,0},c_{2,0}}(c_1',c_2',x_1,x_2,y)\\
    \times e\left(\alpha(Nx_1)^\beta-\frac{n_1Nx_1}{c_0c_{1,0}c_1'}-\alpha(Nx_2)^\beta+\frac{n_2Nx_2}{c_0c_{2,0}c_2'}\right)dx_1dx_2dy,
\end{multline*}
where 
\begin{align*}
    \Phi_{m,n_1,n_2}^{c_0,c_{1,0},c_{2,0}}(c_1',c_2',x_1,x_2) &\coloneqq T^{-2\epsilon}\left(\frac{m_0}{C_0}\right)^{2\epsilon \eta} V_{0} \left(\frac{bc_1m_0}{C}, u \right)V_{0} \left(\frac{bc_2m_0}{C}, u \right)\\
    &\quad \times e\left(\mp\frac{c_{2,0}c_2'}{c_{1,0}c_1'mn_1}\mp\frac{c_{1,0}c_1'}{c_{2,0}c_2'mn_2}-\frac{mLMy}{c_1c_2}\right)\\
    &\quad \times \Phi\left(x_1+\frac{bc_1m_0Cu}{N}, \frac{LMy}{c_1^3m_0} \right)\overline{\Phi\left(x_2+\frac{bc_2m_0Cu}{N}, \frac{LMy}{c_2^3m_0} \right)}\\
    &\quad \times \left(\frac{C_0^3\left(x_1+\frac{bc_1m_0Cu}{N}\right)}{c_1^3m_0^3}\right)^{(\frac{1}{6}+\frac{\epsilon}{3}-\frac{it}{3})\eta} \left(\frac{C_0^3\left(x_2+\frac{bc_2m_0Cu}{N}\right)}{c_2^3m_0^3}\right)^{(\frac{1}{6}+\frac{\epsilon}{3}+\frac{it}{3})\eta}
\end{align*}
is a $T^\epsilon$-inert function satisfying $\Phi_{m,n_1,n_2}^{c_0,c_{1,0},c_{2,0}}(c_1',c_2',x_1,x_2)\ll 1$.

It is also convenient to remove the coprimality condition $(c_1',c_2')=1$ via M\"{o}bius inversion:
\begin{align*}
    \mathcal{R}(N) &\ll \frac{LM^{\frac{4}{3}} T^\epsilon}{N^{\frac{2}{3}}}+LM\sum_{0\neq |m|\ll \frac{C_0^2T^\epsilon}{m_0^2LM}}\sum_{c_0 \mid m}\mathop{\sum\sum}_{\substack{c_{1,0},c_{2,0} \mid c_0^\infty\\ (c_{1,0},c_{2,0})=1}}\mathop{\sum\sum\sum}_{\substack{\frac{C_0}{c_0m_0}\leq ac_{1,0}c_1',ac_{2,0}c_2'<\frac{2C_0}{c_0m_0}\\(ac_1'c_2',m)=1}} \frac{\mu(a)}{a^2c_0^2c_{1,0}c_{2,0}c_1'c_2'}\\
    & \quad \times \mathop{\sum\sum}_{\substack{\frac{C_0T^{-\epsilon}}{N^{1-\beta}}\ll n_1,n_2\ll \frac{C_0T^\epsilon}{N^{1-\beta}}\\(n_1,ac_0c_1'm_0)=(n_2,ac_0c_2'm_0)=1}} \mathop{\sumast_{\alpha_1\tpmod{c_0c_{1,0}}}\sumast_{\alpha_2\tpmod{c_0c_{2,0}}}}_{ \alpha_1c_{2,0}c_2'+\alpha_2c_{1,0}c_1'\equiv \mp \frac{m}{c_0}\overline{a}\tpmod{c_0c_{1,0}c_{2,0}}}e\left(\frac{\overline{\alpha_1ac_1'n_1}}{c_0c_{1,0}}+\frac{\overline{\alpha_2ac_2'n_2}}{c_0c_{2,0}}\right)\\
    & \quad \times e\left(\pm\frac{c_{2,0}c_2'\overline{c_1'}}{c_{1,0}mn_1}\pm\frac{c_{1,0}c_1'\overline{c_2'}}{c_{2,0}mn_2}\right)\mathcal{J}_5^\eta(m,n_1,n_2,c_0,c_{1,0},c_{2,0},ac_1',ac_2').
\end{align*}

\subsection{Cauchy--Schwarz in \texorpdfstring{$c_{1}'$}{c1'}}

We now take advantage of the $(c_1',c_2')$-sums. To simplify the forthcoming analysis, we pull out common divisors of $n_1$ and $n_2$. Write $n_0=(n_1,n_2)$,~$n_{1,0}=(n_1,n_0^\infty)$, $n_{2,0}=(n_2,n_0^\infty)$, $(n_1/(n_0n_{1,0},m^\infty))=m_1$, $(n_2/(n_0n_{2,0},m^\infty))=m_2$, $n_1=m_1n_0n_{1,0}n_1'$ and $n_2=m_2n_0n_{2,0}n_2'$. Under this notation, our estimate transforms into
\begin{align*}
    \mathcal{R}(N) &\ll \frac{LM^{\frac{4}{3}} T^\epsilon}{N^{\frac{2}{3}}}+LM\sum_{0\neq |m|\ll \frac{C_0^2T^\epsilon}{m_0^2LM}}\sum_{c_0 \mid m}\mathop{\sum\sum}_{\substack{c_{1,0},c_{2,0} \mid c_0^\infty\\ (c_{1,0},c_{2,0})=1}}\mathop{\sum\sum\sum}_{\substack{\frac{C_0}{c_0m_0}\leq ac_{1,0}c_1',ac_{2,0}c_2'<\frac{2C_0}{c_0m_0}\\(ac_1'c_2',m)=1}} \frac{\mu(a)}{a^2c_0^2c_{1,0}c_{2,0}c_1'c_2'}\\
    & \quad \times \sum_{(n_0,ac_0c_1'c_2'm_0)=1}\mathop{\sum\sum}_{\substack{n_{1,0},n_{2,0}\mid n_0^\infty\\ (n_{1,0},n_{2,0})=1}}\mathop{\sum\sum}_{\substack{m_1,m_2 \mid m^\infty\\(m_1,m_2)=(m_1m_2,ac_0m_0n_0)=1\\(m_1,c_1')=(m_2,c_2')=1}}\mathop{\sum\sum}_{\substack{\frac{C_0T^{-\epsilon}}{N^{1-\beta}}\ll n_1=m_1n_0n_{1,0}n_1'\ll \frac{C_0T^\epsilon}{N^{1-\beta}}\\
    \frac{C_0T^{-\epsilon}}{N^{1-\beta}}\ll n_2=m_2n_0n_{2,0}n_2'\ll \frac{C_0T^\epsilon}{N^{1-\beta}}\\(n_1',n_2')=(n_1'n_2',ac_0m_0mn_0)=1\\ (n_1',c_1')=(n_2',c_2')=1}} \\
    & \quad \times \mathop{\sumast_{\alpha_1\tpmod{c_0c_{1,0}}}\sumast_{\alpha_2\tpmod{c_0c_{2,0}}}}_{ \alpha_1c_{2,0}c_2'+\alpha_2c_{1,0}c_1'\equiv \mp \frac{m}{c_0}\overline{a}\tpmod{c_0c_{1,0}c_{2,0}}}e\left(\frac{\overline{\alpha_1ac_1'n_1}}{c_0c_{1,0}}+\frac{\overline{\alpha_2ac_2'n_2}}{c_0c_{2,0}}\pm\frac{c_{2,0}c_2'\overline{c_1'}}{c_{1,0}mn_1}\pm\frac{c_{1,0}c_1'\overline{c_2'}}{c_{2,0}mn_2}\right)\\
    &\quad \times \mathcal{J}_5^\eta(m,n_1,n_2,c_0,c_{1,0},c_{2,0},ac_1',ac_2').
\end{align*}
Note that the second derivative test with the choice $Y\leq N^\beta T^{-\epsilon}$ implies that for any $T^\epsilon$-inert function $F$,
\begin{equation*}
    \int_0^\infty W(x)F(x)e\left(\alpha(Nx_1)^\beta-\frac{n_1Nx}{ac_0c_{1,0}c_1'}\right)dx \ll N^{-\frac{\beta}{2}}T^\epsilon.
\end{equation*}
By the Cauchy--Schwarz inequality to take out the $(a,c_0,c_{1,0},c_{2,0},c_1',\alpha_1,\alpha_2,m,n_0,n_{1,0},n_{2,0},n_1',n_2')$-sums and the $y$-integral, along with a lengthening parameter $R_1 \geq 1$, we deduce
\begin{align*}
    \mathcal{R}(N) &\ll \frac{LM^{\frac{4}{3}} T^\epsilon}{N^{\frac{2}{3}}}+\frac{LMT^\epsilon}{N^{\frac{\beta}{2}}}\left(\frac{(MN)^{\frac{1}{3}}}{C_0}\right)^\eta\Bigg\{\sum_{0\neq |m|\ll \frac{C_0^2T^\epsilon}{m_0^2LM}}\sum_{(a,m)=1}\sum_{c_0 \mid m}\mathop{\sum\sum}_{\substack{c_{1,0},c_{2,0} \mid c_0^\infty\\(c_{1,0},c_{2,0})=1}}\\
    &\quad \times \sum_{(n_0,ac_0c_1'c_2'm_0)=1}\mathop{\sum\sum}_{\substack{n_{1,0},n_{2,0}\mid n_0^\infty\\ (n_{1,0},n_{2,0})=1}}\mathop{\sum\sum}_{\substack{m_1,m_2 \mid m^\infty\\(m_1,m_2)=(m_1m_2,ac_0m_0n_0)=1}}\mathop{\sum\sum}_{\substack{\frac{C_0T^{-\epsilon}}{N^{1-\beta}}\ll n_1=m_1n_0n_{1,0}n_1'\ll \frac{C_0T^\epsilon}{N^{1-\beta}}\\
    \frac{C_0T^{-\epsilon}}{N^{1-\beta}}\ll n_2=m_2n_0n_{2,0}n_2'\ll \frac{C_0T^\epsilon}{N^{1-\beta}}\\(n_1',n_2')=(n_1'n_2',ac_0m_0mn_0)=1}}\\
    &\quad \times  \mathop{\sumast_{\alpha_1\tpmod{c_0c_{1,0}}}\sumast_{\alpha_2\tpmod{c_0c_{2,0}}}}\sum_{(c_1',mn_0n_1')=1}\frac{m_0mn_0n_1'n_2'}{a^2c_0^2c_{1,0}c_{2,0}^2C_0\sqrt{(mn_0)_\square}}U_0\left(\frac{ac_0c_{1,0}c_1'm_0}{C_0R_1}\right)\\
    &\quad \times \Bigg|\sum_{\substack{\frac{C_0}{c_0m_0}\leq ac_{2,0}c_2'<\frac{2C_0}{c_0m_0}\\(c_2',mn_0n_2')=1\\\alpha_1c_{2,0}c_2'+\alpha_2c_{1,0}c_1'\equiv \mp \frac{m}{c_0}\overline{a}\tpmod{c_0c_{1,0}c_{2,0}}}}\frac{1}{c_2'}e\left(\frac{\overline{\alpha_2ac_2'n_2}}{c_0c_{2,0}}\pm\frac{c_{2,0}c_2'\overline{c_1'}}{c_{1,0}mn_1}\pm\frac{c_{1,0}c_1'\overline{c_2'}}{c_{2,0}mn_2}\right)\\
    &\quad \times \int_{0}^\infty  W(x)\varphi(c_1',c_2',x)e\left(-\alpha(Nx)^\beta+\frac{n_2Nx}{ac_0c_{2,0}c_2'}\right)dx\Bigg|^2\Bigg\}^{\frac{1}{2}}
\end{align*}
for some $T^\epsilon$-inert function $\varphi(c_1',c_2',x)\ll 1$.

\subsection{Poisson Summation in \texorpdfstring{$c_{1}'$}{c1'}}

Opening the square, the $c_1'$-sum is now given by
\begin{equation*}
    \sum_{(c_1',mn_0n_1')=1}U_0\left(\frac{ac_0c_{1,0}c_1'm_0}{C_0R_1}\right)\varphi(c_1',c_2',x_2)\overline{\varphi(c_1',c_3',x_3)}e\left(\pm\frac{c_{2,0}\overline{c_1'}(c_2'-c_3')}{c_{1,0}mn_1}\pm\frac{c_{1,0}c_1'(\overline{c_2'}-\overline{c_3'})}{c_{2,0}mn_2}\right).
\end{equation*}
By Poisson summation (Lemma~\ref{Poisson-summation}), this is equal to
\begin{multline*}
    \frac{n_0C_0R_1}{ac_0c_{1,0}^2c_{2,0}m_0mn_1n_2}\sum_{c_1'}\mathcal{C}'(c_1',c_2',c_3')\int_\R U_0(y)\varphi\left(\frac{C_0R_1y}{ac_0c_{1,0}m_0},c_2',x_2\right)\\
    \times \overline{\varphi\left(\frac{C_0R_1y}{ac_0c_{1,0}m_0},c_3',x_3\right)}e\left(-\frac{c_1'n_0C_0R_1y}{ac_0c_{1,0}^2c_{2,0}m_0mn_1n_2}\right)dy,
\end{multline*}
where
\begin{equation*}
    \mathcal{C}'(c_1',c_2',c_3') \coloneqq \sum_{\substack{\gamma\tpmod{c_{1,0}c_{2,0}mn_1n_2/n_0}\\(\gamma,mn_0n_1')=1}}e\left(\pm\frac{c_{2,0}\overline{\gamma}(c_2'-c_3')}{c_{1,0}mn_1}\pm\frac{c_{1,0}\gamma(\overline{c_2'}-\overline{c_3'})}{c_{2,0}mn_2}+\frac{c_1'n_0\gamma}{c_{1,0}c_{2,0}mn_1n_2}\right).
\end{equation*}
Repeated integration by parts in the $y$-integral ensures an arbitrary saving unless
\begin{equation*}
    |c_1'|\ll \frac{ac_0c_{1,0}^2c_{2,0}m_0mn_1n_2}{n_0C_0R_1}T^\epsilon \ll \frac{ac_0c_{1,0}^2c_{2,0}m_0mC_0}{n_0N^{2-2\beta}R_1}T^\epsilon.
\end{equation*}
Choosing $R_1=1+\frac{ac_0c_{1,0}^2c_{2,0}m_0mC_0}{n_0N^{2-2\beta}}T^{2\epsilon}$ ensures an arbitrary saving unless $c_1'=0$, in which~case
\begin{align*}
    \mathcal{C}'(0, c_{2}^{\prime}, c_{3}^{\prime}) &= \sum_{\substack{\gamma \tpmod{c_{1, 0} c_{2, 0} m n_{1} n_{2} /n_0}\\(\gamma,mn_0n_1')=1}} e \left(\frac{c_{2, 0} \overline{\gamma}(c_{2}^{\prime}-c_{3}^{\prime})}{c_{1, 0}mn_1}+\frac{c_{1, 0} \gamma(\overline{c_{2}^{\prime}}-\overline{c_{3}^{\prime}})}{c_{2, 0} mn_2} \right)\\
    &= n_2'\delta(c_2'\equiv c_3'\tpmod{n_2'})\\
    &\quad \times S(c_{1,0}^2m_1n_{1,0}n_1'\overline{n_2'}(\overline{c_{2}^{\prime}}-\overline{c_{3}^{\prime}}),c_{2,0}^2m_2n_{2,0}(c_{2}^{\prime}-c_{3}^{\prime});c_{1,0}c_{2,0}m_1m_2mn_0n_{1,0}n_{2,0}n_1').
\end{align*}
Because $c_{1,0} \mid c_0^\infty \mid m^\infty$, we observe that for any $p \mid c_{1,0}$, the Kloosterman sum evaluates to $0$ unless the $p$-adic valuation obeys
\begin{equation*}
     v_p(c_{1,0}^2(\overline{c_{2}^{\prime}}-\overline{c_{3}^{\prime}}))=v_p(c_{2}^{\prime}-c_{3}^{\prime}), \qquad c_{1,0}(c_{1,0}^\infty,m) \mid c_{2}^{\prime}-c_{3}^{\prime}.
\end{equation*}
By $p^k \mid c_{2}^{\prime}-c_{3}^{\prime} \Leftrightarrow p^k \mid \overline{c_{2}^{\prime}}-\overline{c_{3}^{\prime}}$, the former condition cannot hold unless $c_2^\prime=c_3^\prime$. Therefore, the Kloosterman sum evaluates to $0$ unless $c_{1,0}(c_{1,0}^\infty,m) \mid c_{2}^{\prime}-c_{3}^{\prime}$. Similarly, we obtain the condition $c_{2,0}m_1m_2n_{1,0}n_{2,0}(mn_0,(c_{2,0}m_1m_2n_{1,0}n_{2,0})^\infty) \mid c_{2}^{\prime}-c_{3}^{\prime}$. Writing $d=(mn_0,(c_{2,0}m_1m_2n_{1,0}n_{2,0})^\infty)$, one derives
\begin{multline*}
    \mathcal{C}'(0, c_{2}^{\prime}, c_{3}^{\prime}) =\varphi(d)c_{1,0}c_{2,0}m_1m_2n_{1,0}n_{2,0}n_2' \delta(c_2'\equiv c_3'\tpmod{c_{1,0}c_{2,0}dm_1m_2n_{1,0}n_{2,0}n_2'})\\
    \times S\left(\frac{c_{1,0}n_1(\overline{c_{2}^{\prime}}-\overline{c_{3}^{\prime}})}{c_{2,0}dm_2n_{2,0}n_2'},\frac{c_{2,0}(c_{2}^{\prime}-c_{3}^{\prime})}{c_{1,0}dm_1n_{1,0}};\frac{mn_0n_1'}{d}\right),
\end{multline*}
where $\varphi(d)$ denotes Euler's totient function. Since $(mn_0,n_1)=1$, we arrive at
\begin{multline*}
    \mathcal{C}'(0, c_{2}^{\prime}, c_{3}^{\prime}) =\varphi(d)c_{1,0}c_{2,0}m_1m_2n_{1,0}n_{2,0}n_2'\\
    \times \sum_{b_0b_1=n_1'}\mu(b_0)b_1\delta(c_2'\equiv c_3'\tpmod{b_1c_{1,0}c_{2,0}dm_1m_2n_{1,0}n_{2,0}n_2'})\\
    \times S\left(\frac{c_{1,0}(\overline{c_{2}^{\prime}}-\overline{c_{3}^{\prime}})}{c_{2,0}dm_2n_{2,0}n_2'},\frac{c_{2,0}\overline{b_0}(c_{2}^{\prime}-c_{3}^{\prime})}{b_1c_{1,0}dm_1n_{1,0}};\frac{mn_0}{d}\right).
\end{multline*}
On the other hand, there holds
\begin{multline*}
    \begin{dcases}\alpha_1c_{2,0}c_2'+\alpha_2c_{1,0}c_1'\equiv \mp \frac{m}{c_0}\overline{a}\tpmod{c_0c_{1,0}c_{2,0}},\\
    \alpha_1c_{2,0}c_3'+\alpha_2c_{1,0}c_1'\equiv \mp \frac{m}{c_0}\overline{a}\tpmod{c_0c_{1,0}c_{2,0}}\end{dcases}\\
    \Longleftrightarrow \quad \begin{dcases}\alpha_1c_{2,0}c_2'+\alpha_2c_{1,0}c_1'\equiv \mp \frac{m}{c_0}\overline{a}\tpmod{c_0c_{1,0}c_{2,0}},\\
    c_2'\equiv c_3'\tpmod{c_0c_{1,0}}.\end{dcases}
\end{multline*}
Because $c_0 \mid m$ and $(m,c_{2,0}^\infty) \mid d$, we have that $c_2'\equiv c_3'\tpmod{c_0c_{1,0}c_{2,0}}$.

Gathering the above observations together yields
\begin{align*}
    \mathcal{R}(N) &\ll \frac{LM^{\frac{4}{3}} T^\epsilon}{N^{\frac{2}{3}}}+\frac{LMT^\epsilon}{N^{\frac{\beta}{2}}}\left(\frac{(MN)^{\frac{1}{3}}}{C_0}\right)^\eta\Bigg\{\sum_{0\neq |m|\ll \frac{C_0^2T^\epsilon}{m_0^2LM}}\sum_{(a,m)=1}\sum_{c_0 \mid m}\mathop{\sum\sum}_{\substack{c_{1,0},c_{2,0} \mid c_0^\infty\\(c_{1,0},c_{2,0})=1}}\\
    &\quad \times \sum_{(n_0,ac_0c_1'c_2'm_0)=1}\mathop{\sum\sum}_{\substack{n_{1,0},n_{2,0}\mid n_0^\infty\\ (n_{1,0},n_{2,0})=1}}\mathop{\sum\sum}_{\substack{m_1,m_2 \mid m^\infty\\(m_1,m_2)=(m_1m_2,ac_0m_0n_0)=1}}\mathop{\sum\sum}_{\substack{\frac{C_0T^{-\epsilon}}{N^{1-\beta}}\ll n_1=m_1n_0n_{1,0}n_1'\ll \frac{C_0T^\epsilon}{N^{1-\beta}}\\
    \frac{C_0T^{-\epsilon}}{N^{1-\beta}}\ll n_2=m_2n_0n_{2,0}n_2'\ll \frac{C_0T^\epsilon}{N^{1-\beta}}\\(n_1',n_2')=(n_1'n_2',ac_0m_0mn_0)=1}}\\
    &\quad \times  \mathop{\sumast_{\alpha_1\tpmod{c_0c_{1,0}}}\sumast_{\alpha_2\tpmod{c_0c_{2,0}}}}\frac{\varphi(d)m_1m_2n_2'R_1}{a^3c_0^3c_{1,0}^2c_{2,0}^2\sqrt{(mn_0)_\square}}\sum_{b_0b_1=n_1'}\mu(b_0)b_1\\
    &\quad \times \mathop{\sum\sum}_{\substack{\frac{C_0}{ac_0c_{2,0}m_0}\leq c_2',c_3'<\frac{2C_0}{ac_0c_{2,0}m_0}\\(c_2'c_3',mn_0n_2')=1\\\alpha_1c_{2,0}c_2'+\alpha_2c_{1,0}c_1'\equiv \mp \frac{m}{c_0}\overline{a}\tpmod{c_0c_{1,0}c_{2,0}}}}\frac{1}{c_2'c_3'}\mathcal{C}_1(c_2',c_3') \mathcal{I}_1(c_2',\ell)\Bigg\}^{\frac{1}{2}},
\end{align*}
where
\begin{equation*}
    \mathcal{C}_1(c_2',c_3')=\delta(c_2'\equiv c_3'\tpmod{q})S\left(\frac{c_{1,0}(\overline{c_{2}^{\prime}}-\overline{c_{3}^{\prime}})}{c_{2,0}dm_2n_{2,0}n_2'},\frac{c_{2,0}\overline{b_0}(c_{2}^{\prime}-c_{3}^{\prime})}{b_1c_{1,0}dm_1n_{1,0}};\frac{mn_0}{d}\right)
\end{equation*}
with the notation
\begin{equation*}
    d=(mn_0,(c_{2,0}m_1m_2n_{1,0}n_{2,0})^\infty), \qquad q=b_1[c_0,d]c_{1,0}c_{2,0}m_1m_2n_{1,0}n_{2,0}n_2',
\end{equation*}
and
\begin{multline*}
    \mathcal{I}_1(c_2',\ell) \coloneqq \int_\R \int_{0}^\infty \int_{0}^\infty W(x_2)W(x_3)U_0(y)\varphi\left(\frac{C_0R_1y}{ac_0c_{1,0}m_0},b_2c_2',x_2\right)\overline{\varphi\left(\frac{C_0R_1y}{ac_0c_{1,0}m_0},b_2c_3',x_3\right)}\\
    \times e\left(-\alpha(Nx_2)^\beta+\frac{n_2Nx_2}{ac_0c_{2,0}c_2'}+\alpha(Nx_3)^\beta-\frac{n_2Nx_3}{ac_0c_{2,0}c_3'}\right)dx_2dx_3dy.
\end{multline*}

\subsection{Iteration Ad Infinitum}
For any $k \in \N$, we iterate the above process $k$ times with a different lengthening parameter
\begin{equation*}
    R_2=1+\frac{ac_0c_{1,0}c_{2,0}^2m_0mC_0}{n_0N^{2-2\beta}}T^{2\epsilon}.
\end{equation*}
This implies
\begin{align*}
    \mathcal{R}(N) &\ll \frac{LM^{\frac{4}{3}} T^\epsilon}{N^{\frac{2}{3}}}+\frac{LMT^\epsilon}{N^{\beta(1-2^{-k-1})}}\left(\frac{(MN)^{\frac{1}{3}}}{C_0}\right)^\eta\Bigg\{\sum_{0\neq |m|\ll \frac{C_0^2T^\epsilon}{m_0^2LM}}\sum_{(a,m)=1}\sum_{c_0 \mid m}\mathop{\sum\sum}_{\substack{c_{1,0},c_{2,0} \mid c_0^\infty\\ (c_{1,0},c_{2,0})=1}}\\
    &\quad \times  \sum_{(n_0,ac_0c_1'c_2'm_0)=1}\mathop{\sum\sum}_{\substack{n_{1,0},n_{2,0}\mid n_0^\infty\\ (n_{1,0},n_{2,0})=1}}\mathop{\sum\sum}_{\substack{m_1,m_2 \mid m^\infty\\(m_1,m_2)=(m_1m_2,ac_0m_0n_0)=1}}\mathop{\sum\sum}_{\substack{\frac{C_0T^{-\epsilon}}{N^{1-\beta}}\ll n_1=m_1n_0n_{1,0}n_1'\ll \frac{C_0T^\epsilon}{N^{1-\beta}}\\
    \frac{C_0T^{-\epsilon}}{N^{1-\beta}}\ll n_2=m_2n_0n_{2,0}n_2'\ll \frac{C_0T^\epsilon}{N^{1-\beta}}\\(n_1',n_2')=(n_1'n_2',ac_0m_0mn_0)=1}}\\
    &\quad \times  \sum_{b_0b_1=n_1'} \frac{(b_1\varphi(d)m_1m_2n_2'R_1)^{2^k}R_2^{2^k-1}}{a^{2^{k+1}+1}c_0^{2^{k+1}}(c_{1,0}c_{2,0})^{3 \cdot 2^{k}-1}(n_{1,0}n_{2,0})^{2^k-1}(mn_0)_\square^{2^k-\frac{1}{2}}} \\
    &\quad \times   \mathop{\sum\sum}_{\substack{\frac{C_0}{ac_0c_{2,0}m_0}\leq c_{k+2}',c_{k+3}'<\frac{2C_0}{ac_0c_{2,0}m_0}\\(c_{k+2}'c_{k+3}',mn_0n_2')=1}}\frac{1}{c_{k+2}'c_{k+3}'}\mathcal{C}_{k+1}(c_{k+2}',c_{k+3}') \mathcal{I}_{k+1}(c_{k+2}',c_{k+3}')\Bigg\}^{\frac{1}{2^{k+1}}},
\end{align*}
where $\mathcal{C}_{k+1}(A, B)$ is defined recursively by
\begin{align*}
    \mathcal{C}_{k+1}(A, B) &\coloneqq \sum_{\substack{\gamma\tpmod{c_{1,0}c_{2,0}mn_1n_2/n_0}\\(\gamma,mn_0n_1')=1}}\mathcal{C}_{k}(\gamma,A)\mathcal{C}_{k}(\gamma,B),\\
    \mathcal{C}_1(A, B) &\coloneqq \delta(A\equiv B\tpmod{q})S\left(\frac{c_{1,0}(\overline{A}-\overline{B})}{c_{2,0}dm_2n_{2,0}n_2'},\frac{c_{2,0}\overline{b_0}(A-B)}{b_1c_{1,0}dm_1n_{1,0}};\frac{mn_0}{d}\right),
\end{align*}
and for $k \in \N$,
\begin{multline*}
    \mathcal{I}_{k+1}(A, B) \coloneqq 
        \int_\R \int_{0}^\infty \int_{0}^\infty W(x_1)W(x_2)U_0(y)\varphi_{k+1}\left(\frac{C_0Ry}{ac_0c_{1,0}m_0},A,x_1\right)\overline{\varphi_{k+1}\left(\frac{C_0Ry}{ac_0c_{1,0}m_0},B,x_2\right)}\\
        \times  e\left((-1)^k\left(\alpha(Nx_1)^\beta-\frac{n_2Nx_1}{ac_0c_{2,0}A}-\alpha(Nx_2)^\beta+\frac{n_2Nx_2}{ac_0c_{2,0}B}\right)\right)dx_1dx_2dy
\end{multline*}
with $\varphi_{k+1}=\varphi$ if $k$ is even and $\overline{\varphi}$ if $k$ is odd.

\subsection{Endgame}
By Lemma \ref{CharSumBound}, we obtain
\begin{equation*}
    \mathcal{C}_{k+1}(c_{k+2}',c_{k+3}')\ll \left(\frac{b_0mn_0}{[c_0,d]}C_1(C_2,c_{1,0}c_{2,0})\left(\frac{mn_0}{d},q^\infty\right)\right)^{2^k},
\end{equation*}
where
\begin{equation*}
C_1= \left(\frac{mn_0}{(mn_0,q^\infty)} \right)_\square, \qquad C_2=\frac{mn_0}{\left(mn_0,q^\infty\right)C_1}.
\end{equation*}
Taking $k$ sufficiently large and bounding everything else trivially yields
\begin{align}\label{eq:R-final}
    \begin{split}
    \mathcal{R}(N) &\ll \frac{LM^{\frac{4}{3}} T^\epsilon}{N^{\frac{2}{3}}}+LMT^\epsilon\frac{R}{N^\beta}\left(\frac{(MN)^{\frac{1}{3}}}{C_0}\right)^\eta\\
    &\quad \times \sup_{\cdots} \cdots \sup_{\cdots} \frac{\sqrt{b_1dm_1m_2n_2'R_1R_2}}{ac_0(c_{1,0}c_{2,0})^{\frac{3}{2}}\sqrt{n_{1,0}n_{2,0}(mn_0)_\square}}\sqrt{\frac{b_0mn_0}{[c_0,d]}C_1(C_2,c_{1,0}c_{2,0})\left(\frac{mn_0}{d},q^\infty\right)}\\
    &\ll \frac{LM^{\frac{4}{3}} T^\epsilon}{N^{\frac{2}{3}}}+\frac{C_0^2\sqrt{LM}T^\epsilon}{m_0N}\left(\frac{(MN)^{\frac{1}{3}}}{C_0}\right)^\eta\left(1+\frac{C_0^3}{m_0LMN^{2-2\beta}}\right).
    \end{split}
\end{align}
It now follows from~\eqref{RNSplit} and~\eqref{eq:R-final} that
\begin{align*}
    \mathcal{S}(N)\ll T^\epsilon\sup_{M\ll\frac{N^2T^\epsilon}{C^3}}\left(\sqrt{L}M^{\frac{5}{6}}N^{\frac{1}{3}}+CL^{\frac{1}{4}}M^{\frac{5}{12}}N^{\frac{1}{6}}\left(1+\frac{(MN)^{\frac{1}{6}}}{\sqrt{C}}\right)\left(1+\frac{C^{\frac{3}{2}}}{\sqrt{LM}N^{1-\beta}}\right)\right),
    \end{align*}
where $L\geq1$, $Y\leq N^\beta T^{-\epsilon}$, and $C\leq \sqrt{N}$. If we choose
\begin{equation*}
L=\left(\frac{N^2T^\epsilon}{C^3M}\right)^{\frac{1}{3}},
\end{equation*}
then the above estimate simplifies to
\begin{equation*}
    \mathcal{S}(N)\ll \left(\frac{N^{2}}{C^{\frac{5}{2}}}+\frac{N^{\frac{3}{2}}}{C^{\frac{5}{4}}}+C^{\frac{7}{4}}N^{\beta-\frac{1}{2}}\right) T^\epsilon.
\end{equation*}
The optimal choice of $C$ is given by
\begin{equation*}
C = \min\{N^{\frac{2-\beta}{3}},\sqrt{N}\}.
\end{equation*}
Altogether, we arrive at
\begin{equation}\label{eq:S-final}
    \mathcal{S}(N)\ll (N^{\frac{2}{3}+\frac{5\beta}{12}}+N^{\frac{7}{8}})T^\epsilon,
\end{equation}
where $\beta < \frac{4}{5}$. This completes the proof of Theorem \ref{main-2}. Moreover, inserting~\eqref{eq:S-final} into~\eqref{DyadicSmoothing} and choosing $Y=N^\beta T^{-\epsilon}$, we conclude that
\begin{equation*}
    \mathcal{A}_\pi(T, \alpha, \beta) \ll T^{\frac{2}{3}+\frac{5\beta}{12}+\epsilon}+T^{\frac{7}{8}+\varepsilon}+T^{\frac{19}{14}-\beta+\epsilon}
\end{equation*}
where $\beta < \frac{4}{5}$. The proof of Theorem~\ref{main} is now complete. \qed

\appendix

\section{Computations of Character Sums}

To handle the character sum that we encounter, we need the following lemmata.
\begin{lemma}\label{lem.charsum1}
    Let $A, B, C, k, u, v \in \N$ with $C$ squarefree. Consider the character sum defined recursively by \begin{align*}
        \mathcal{C}_1(A, B) &= S(u(\overline{A}-\overline{B}),v(A-B);C),\\
        \mathcal{C}_{k+1}(A, B) &= \sumast_{\gamma\tpmod{C}}\mathcal{C}_{k}(\gamma,A)\mathcal{C}_{k}(\gamma,B).
    \end{align*}
    Then we have that
    \begin{align*}
        \mathcal{C}_{k}(A, B) \ll ((C,uv)C)^{2^{k-1}-\frac{1}{2}}\sqrt{(C,A-B)}.
    \end{align*}
\end{lemma}

\begin{proof}
    We first assume that $C=p$ is prime. If $p \mid uv$, then trivial bound $\mathcal{C}_k\ll p^{2^k-1}$ follows. On the other hand, if $p\nmid uv(A-B)$, then we apply the method of Adolphson--Sperber~\cite{AdolphsonSperber1993}. For $k=1$, the Weil bound for Kloosterman sums implies $\mathcal{C}_1(A, B)\ll \sqrt{p}$. For $k=2$,
    \begin{align*}
    \mathcal{C}_2(A, B)=\mathop{\sumast \ \sumast \ \sumast}_{\alpha_1,\alpha_2,\beta\tpmod{p}}e\left(\frac{f_2(\alpha_1,\alpha_2,\beta)}{p}\right),
    \end{align*}
    where \begin{equation*}
        f_2(\alpha_1,\alpha_2,\beta)=u\alpha_1\overline{\beta}+v\overline{\alpha_1}\beta-u\alpha_1\overline{A}-v\overline{\alpha_1}A
        +u\alpha_2\overline{\beta}+v\overline{\alpha_2}\beta-u\alpha_2\overline{B}-v\overline{\alpha_2}B.
    \end{equation*}
    Considering $f_2(x,y,z)\in \mathbb{F}_p^{\times}[x,y,z,(xyz)^{-1}]$, we see that the locus of $\frac{\partial f_{2,\sigma}}{\partial x} = \frac{\partial f_{2,\sigma}}{\partial y} = \frac{\partial f_{2,\sigma}}{\partial z} = 0$ is empty in $(\mathbb{F}_p^{\times})^{3}$ for any face $\sigma$ of the Newton polyhedron $\Delta(f)$ since $p\nmid uv(A-B)$. Hence the work of Adolphson--Sperber~\cite{AdolphsonSperber1993} yields $\mathcal{C}_2(A, B)\ll p^{\frac{3}{2}}$.
    
    For $k\geq3$, we observe that the character sum $\mathcal{C}_{k}(A, B)$ can be written in the form
    \begin{align*}
        \mathcal{C}_{k}(A, B)=\mathop{\sumast\cdots\sumast}_{\substack{\alpha_1,\ldots,\alpha_{2^{k-1}}\tpmod{p}\\\beta_1,\ldots,\beta_{2^{k-2}}\tpmod{p}\\\gamma_1,\ldots,\gamma_{2^{k-2}-1}\tpmod{p}}}e\left(\frac{f_{k}(\alpha_1,\ldots,\alpha_{2^{k-1}},\beta_1,\ldots,\beta_{2^{k-2}},\gamma_1,\ldots,\gamma_{2^{k-2}-1})}{p}\right),
    \end{align*}
    where
    \begin{multline*}
        f_k(\alpha_1,\ldots,\alpha_{2^{k-1}},\beta_1,\ldots,\beta_{2^{k-2}},\gamma_1,\ldots,\gamma_{2^{k-2}-1})\\
        =\sum_{j=1}^{2^{k-1}}\left(u\alpha_j\overline{\beta_{\lceil \frac{j}{2}\rceil}}+v\overline{\alpha_j}\beta_{\lceil \frac{j}{2}\rceil}-u\alpha_j\overline{\gamma_{g(j)}}-v\overline{\alpha_j}\gamma_{g(j)}\right)
        -u\alpha_{2^{k-2}}\overline{A}-u\overline{\alpha_{2^{k-2}}}A-v\alpha_{2^{k-1}}\overline{B}-v\overline{\alpha_{2^{k-1}}}B.
    \end{multline*}
    For $j= 2^\ell+2^{\ell+2}h, \, 2^\ell+2^{\ell+1}(2h+1)$ with $\ell,h\geq0$, the function $g(j)$ above is given by
    \begin{align*}
        g(j)=2^{k-2}-2^{k-\ell-2}+h.
    \end{align*}
    Considering $f_k(x_1,\ldots,x_{2^{k}-1})\in \mathbb{F}_p^{\times}[x_1,\ldots,x_{2^{k}-1},(x_1\cdots x_{2^{k}-1})^{-1}]$, we see that locus of $\frac{\partial f_{k,\sigma}}{\partial x_j} = 0$ for all $1\leq j\leq 2^{k}-1$ is empty in $(\mathbb{F}_p^{\times})^{2^{k}-1}$ for any face $\sigma$ of the Newton polyhedron $\Delta(f_k)$ since $p\nmid uv(A-B)$. Hence the work of Adolphson--Sperber~\cite{AdolphsonSperber1993} yields $\mathcal{C}_k(A, B)\ll p^{2^{k-1}-\frac{1}{2}}$.

    Finally, we perform induction to prove that
    \begin{align*}
        \mathcal{C}_k(A, B)\ll p^{2^{k-1}-\frac{1}{2}}\sqrt{(p,A-B)}
    \end{align*}
    when $p\nmid uv$. We are left with the case where $p \mid (A-B)$. Note that when $k=1$, the Weil bound for Kloosterman sums justifies the desired claim. For $k\geq2$, we have \begin{align*}
        \mathcal{C}_k(A, B)=\sumast_{\gamma\tpmod{p}}\mathcal{C}_{k-1}(\gamma,A)\mathcal{C}_{k-1}(\gamma,B).
    \end{align*}
    By induction hypothesis, \begin{align*}
        \mathcal{C}_{k-1}(\gamma,A)\ll p^{2^{k-2}-\frac{1}{2}}\sqrt{(p,\gamma-A)}.
    \end{align*}
    Therefore, we obtain
    \begin{align*}
        \mathcal{C}_k(A, B)\ll p^{2^{k-1}}=p^{2^{k-1}-\frac{1}{2}}\sqrt{(p,A-B)}.
    \end{align*}
    The Chinese Remainder Theorem completes the proof of Lemma~\ref{lem.charsum1} for squarefree $C$.
\end{proof}

\begin{lemma}\label{CharSumBound}
    Let $A, B, C, k, u, v, q_1, q_2, q_3, q, Q \in \N$ be such that $C \mid Q$, $q_1,q_2,q_3 \mid q \mid Q$, and $(q_3, \frac{CQ}{q_3})=1$. Consider the character sum defined recursively by
    \begin{align*}
        \mathcal{C}_1(A, B) &=\delta(A\equiv B\tpmod{q})S\left(\frac{u(\overline{A}-\overline{B})}{q_1},\frac{v(A-B)}{q_2};C\right),\\
        \mathcal{C}_{k+1}(A, B) &= \sum_{\substack{\gamma\tpmod{Q}\\(\gamma, \frac{Q}{q_3})=1}}\mathcal{C}_{k}(\gamma,A)\mathcal{C}_{k}(\gamma,B).
    \end{align*}
    Then we have that
    \begin{align*}
        \mathcal{C}_k(A, B)\ll \left(\frac{Q}{q}C_1(C_2,uv)(C,q^\infty)\right)^{2^{k-1}},
    \end{align*}
    where $C_1C_2=\frac{C}{(C,q^\infty)}$ with $C_1$ squarefull, $C_2$ squarefree, and $(C_1,C_2)=1$.
\end{lemma}
\begin{proof}
    Write $C_0=(C,q^\infty)$, $C=C_0C'$, $Q_0=(Q,q^\infty)$ and $Q=Q_0Q'$. Then
    \begin{align*}
        \mathcal{C}_k(A, B)=\mathcal{C}_{k,1}(A, B)\mathcal{C}_{k,2}(A, B),
    \end{align*}
    where \begin{align*}
        \mathcal{C}_{1,1}(A, B)&=\delta(A\equiv B\tpmod{q})S\left(\overline{C'^2}\frac{u(\overline{A}-\overline{B})}{q_1},\frac{v(A-B)}{q_2};C_0\right),\\
        \mathcal{C}_{1,2}(A, B)&=S(\overline{C_0^2q_1q_2}u(\overline{A}-\overline{B}),v(A-B);C'),\\
        \mathcal{C}_{k+1,1}(A, B)&=\sum_{\substack{\gamma\tpmod{Q_0}\\(\gamma, \frac{Q_0}{q_3})=1}}\mathcal{C}_{k,1}(\gamma,A)\mathcal{C}_{k,1}(\gamma,B),\\
        \mathcal{C}_{k+1,2}(A, B) &= \sumast_{\gamma\tpmod{Q'}}\mathcal{C}_{k,2}(\gamma,A)\mathcal{C}_{k,2}(\gamma,B)\\
        &=\varphi\left(\frac{Q'}{(Q',C^\infty)}\right)\frac{(Q',C^\infty)}{C'}\sumast_{\gamma\tpmod{C'}}\mathcal{C}_{k,2}(\gamma,A)\mathcal{C}_{k,2}(\gamma,B).
    \end{align*}
    For $\mathcal{C}_{k,1}(A, B)$, bounding the Kloosterman sum trivially implies
    \begin{align*}
        \mathcal{C}_{k,1}(A, B)\ll \left(\frac{Q_0}{q}\right)^{2^{k-1}-1}C_0^{2^{k-1}}.
    \end{align*}
    For $\mathcal{C}_{k,2}(A, B)$, note that the character sum is an iteration of Kloosterman sums. Factorising $C'=C_1C_2$ uniquely with $C_1$ squarefull, $C_2$ squarefree, and $(C_1,C_2)=1$, and then applying the Chinese Remainder Theorem accordingly, Lemma \ref{lem.charsum1} yields
    \begin{align*}
        \mathcal{C}_{k,2}(A, B)\ll \left(\varphi\left(\frac{Q'}{(Q',C^\infty)}\right)\frac{(Q',C^\infty)}{C'}\right)^{2^{k-1}-1}C_1^{2^k-1}\left(\left(C_2,uv\right)C_2\right)^{2^{k-1}-\frac{1}{2}}\sqrt{(C_2,A-B)}.
    \end{align*}
    Combining the above estimates completes the proof of Lemma~\ref{CharSumBound}.
\end{proof}


\begin{thebibliography}{FKM15}

\bibitem[ALM22]{AggarwalLeungMunshi2022}
Keshav Aggarwal, Wing~Hong Leung, and Ritabrata Munshi,
  \emph{\href{https://arxiv.org/abs/2206.06517}{Short Second Moment Bound and
  Subconvexity for {$\mathrm{GL}(3)$} {$L$}-Functions}}, arXiv e-prints (2022),
  58 pages.

\bibitem[AS93]{AdolphsonSperber1993}
Alan~Carl Adolphson and Steven~I. Sperber,
  \emph{\href{https://doi.org/10.1515/crll.1993.443.151}{Twisted Exponential
  Sums and {N}ewton Polyhedra}}, Journal f\"ur die Reine und Angewandte
  Mathematik \textbf{443} (1993), 151--177. \MR{1241131}

\bibitem[Blo12]{Blomer2012}
Valentin Blomer,
  \emph{\href{https://doi.org/10.1353/ajm.2012.0032}{Subconvexity for Twisted
  {$L$}-Functions on {$\mathrm{GL}(3)$}}}, American Journal of Mathematics
  \textbf{134} (2012), no.~5, 1385--1421. \MR{2975240}

\bibitem[Bru06]{Brumley2006}
Farrell Brumley, \emph{\href{https://doi.org/10.1353/ajm.2006.0042}{Effective
  Multiplicity One on {$\mathrm{GL}_{N}$} and Narrow Zero-Free Regions for
  {R}ankin-{S}elberg {$L$}-Functions}}, American Journal of Mathematics
  \textbf{128} (2006), no.~6, 1455--1474. \MR{2275908}

\bibitem[DFI94]{DukeFriedlanderIwaniec1994}
William~Drexel Duke, John~Benjamin Friedlander, and Henryk Iwaniec,
  \emph{\href{https://doi.org/10.1007/BF01231758}{A Quadratic Divisor
  Problem}}, Inventiones Mathematicae \textbf{115} (1994), no.~2, 209--217.
  \MR{1258903}

\bibitem[FKM15]{FouvryKowalskiMichel2015}
\'{E}tienne Fouvry, Emmanuel Kowalski, and Philippe~Gabriel Michel,
  \emph{\href{https://doi.org/10.1112/S0025579314000096}{On the Exponent of
  Distribution of the Ternary Divisor Function}}, Mathematika \textbf{61}
  (2015), no.~1, 121--144. \MR{3333965}

\bibitem[Gol06]{Goldfeld2006}
Dorian~Morris Goldfeld,
  \emph{\href{https://doi.org/10.1017/CBO9780511542923}{Automorphic Forms and
  {$L$}-Functions for the Group {$\mathrm{GL}(n, \mathbf{R})$}}}, Cambridge
  Studies in Advanced Mathematics, vol.~99, Cambridge University Press,
  Cambridge, 2006, With an Appendix by Kevin Alfred Broughan. \MR{2254662}

\bibitem[Har14]{Hardy1914}
Godfrey~Harold Hardy, \emph{Sur les zeros de la fonction {$\zeta(s)$} de
  {R}iemann}, Comptes Rendus de l'Acad\'{e}mie des Sciences \textbf{158}
  (1914), 1012--1014.

\bibitem[HB96]{HeathBrown1996}
Roger Heath-Brown, \emph{\href{https://doi.org/10.1515/crll.1996.481.149}{A New
  Form of the Circle Method, and Its Application to Quadratic Forms}}, Journal
  f\"{u}r die Reine und Angewandte Mathematik \textbf{481} (1996), 149--206.
  \MR{1421949}

\bibitem[Ing27]{Ingham1927}
Albert~Edward Ingham,
  \emph{\href{https://doi.org/10.1112/plms/s2-27.1.273}{Mean-Value Theorems in
  the Theory of the {R}iemann Zeta-Function}}, Proceedings of the London
  Mathematical Society \textbf{27} (1927), no.~4, 273--300. \MR{1575391}

\bibitem[Ivi10]{Ivic2010}
Aleksandar Ivi\'{c}, \emph{\href{https://doi.org/10.2478/s11533-010-0071-y}{On
  Some Problems Involving {H}ardy's Function}}, Central European Journal of
  Mathematics \textbf{8} (2010), no.~6, 1029--1040. \MR{2736904}

\bibitem[Ivi13]{Ivic2013}
\bysame, \emph{\href{https://doi.org/10.1017/CBO9781139236973}{The Theory of
  {H}ardy's {$Z$}-Function}}, Cambridge Tracts in Mathematics, vol. 196,
  Cambridge University Press, Cambridge, 2013. \MR{3838399}

\bibitem[Ivi17a]{Ivic2017-2}
\bysame, \emph{\href{http://mjcnt.phystech.edu/en/article.php?id=128}{On
  Certain Moments of {H}ardy's Function {$Z(t)$} over Short Intervals}}, Moscow
  Journal of Combinatorics and Number Theory \textbf{7} (2017), no.~2, 59--73.
  \MR{3671927}

\bibitem[Ivi17b]{Ivic2017-3}
\bysame, \emph{\href{https://doi.org/10.1007/978-3-319-59969-4_6}{On a Cubic
  Moment of {H}ardy's Function with a Shift}},
  \href{https://doi.org/10.1007/978-3-319-59969-4}{Exploring the {R}iemann Zeta
  Function. {$190$} Years from {R}iemann's Birth} (Hugh~Lowell Montgomery,
  Ashkan Nikeghbali, and Michael~Th. Rassias, eds.), Springer, Cham, 2017, With
  a Preface by Freeman John Dyson, pp.~99--112. \MR{3700040}

\bibitem[Ivi17c]{Ivic2017}
\bysame, \emph{\href{https://doi.org/10.1134/S0371968517010083}{Hardy's
  Function {$Z(t)$}: Results and Unsolved Problems}}, Trudy Matematicheskogo
  Instituta Imeni V. A. Steklova \textbf{296} (2017), 111--122, English Version
  Published in Proceedings of the Steklov Institute of Mathematics \textbf{296}
  (2017), no. 1, 104--114. \MR{3640776}

\bibitem[JS90]{JacquetShalika1990-2}
Herv\'{e}~Michel Jacquet and Joseph~Andrew Shalika, \emph{{R}ankin-{S}elberg
  {C}onvolutions: {A}rchimedean {T}heory}, Festschrift in Honor of Ilya
  Iosifovich Piatetski-Shapiro on the Occasion of His Sixtieth Birthday. Part
  {I}. Papers in Representation Theory. Papers from the Workshop on
  {$L$}-Functions, Number Theory, and Harmonic Analysis Held at Tel-Aviv
  University, Ramat Aviv, May 14--19, 1989 (Stephen~Samuel Gelbart, Roger~Evans
  Howe, and Peter~Clive Sarnak, eds.), Israel Mathematical Conference
  Proceedings, vol.~2, Weizmann Science Press of Israel, Jerusalem, 1990,
  pp.~125--207. \MR{1159102}

\bibitem[Jut87]{Jutila1987}
Matti~Ilmari Jutila, \emph{\href{https://doi.org/10.1007/BF02837820}{On
  Exponential Sums Involving the {R}amanujan Function}}, Proceedings of the
  Indian Academy of Sciences \textbf{97} (1987), no.~1-3, 157--166. \MR{983611}

\bibitem[Jut09]{Jutila2009}
\bysame, \emph{\href{https://doi.org/10.1016/j.jnt.2009.02.011}{Atkinson's
  Formula for {H}ardy's Function}}, Journal of Number Theory \textbf{129}
  (2009), no.~11, 2853--2878. \MR{2549539}

\bibitem[Jut11]{Jutila2011}
\bysame, \emph{\href{https://doi.org/10.1007/s11512-010-0122-4}{An Asymptotic
  Formula for the Primitive of {H}ardy's Function}}, Arkiv f\"{o}r Matematik
  \textbf{49} (2011), no.~1, 97--107. \MR{2784259}

\bibitem[Kim03]{Kim2003}
Henry~Hyeongsin Kim,
  \emph{\href{https://doi.org/10.1090/S0894-0347-02-00410-1}{Functoriality for
  the Exterior Square of {$\mathrm{GL}_{4}$} and the Symmetric Fourth of
  {$\mathrm{GL}_{2}$}}}, Journal of the American Mathematical Society
  \textbf{16} (2003), no.~1, 139--183, With Appendix 1 by Dinakar Ramakrishnan
  and Appendix 2 by Henry Hyeongsin Kim and Peter Clive Sarnak. \MR{1937203}

\bibitem[KL23]{KanekoLeung2023}
Ikuya Kaneko and Wing~Hong Leung, \emph{The {S}hort {S}econd {M}oment of
  {$\mathrm{GL}_{3}$} {$L$}-{F}unctions in the {D}epth {A}spect}, preprint
  (2023), 43 pages.

\bibitem[KMS22]{KumarMalleshamSingh2022}
Sumit Kumar, Kummari Mallesham, and Saurabh~Kumar Singh,
  \emph{\href{https://doi.org/10.1007/s00605-022-01725-x}{Non-Linear Additive
  Twists of {$GL(3) \times GL(2)$} and {$GL(3)$} {M}aass Forms}}, Monatshefte
  f\"{u}r Mathematik \textbf{199} (2022), no.~2, 315--361. \MR{4480813}

\bibitem[Kor08]{Korolev2008}
Maxim~Aleksandrovich Korolev,
  \emph{\href{https://doi.org/10.1070/IM2008v072n03ABEH002407}{On the Integral
  of the {H}ardy Function {$Z(t)$}}}, Izvestiya Rossiiskoi Akademii Nauk.
  Seriya Matematicheskaya \textbf{72} (2008), no.~3, 19--68, translation in
  Izvestiya: Mathematics \textbf{72} (2008), no. 3, 429--478. \MR{2432752}

\bibitem[Leu22]{Leung2022}
Wing~Hong Leung,
  \emph{\href{http://rave.ohiolink.edu/etdc/view?acc_num=osu1650554459410968}{A
  Reformulation of the Delta Method and the Subconvexity Problem}}, {Ph.D.}
  thesis, The Ohio State University, 2022, p.~199. \MR{4495301}

\bibitem[Li11]{Li2011}
Xiaoqing Li, \emph{\href{https://doi.org/10.4007/annals.2011.173.1.8}{Bounds
  for {$\mathrm{GL}(3) \times \mathrm{GL}(2)$} {$L$}-Functions and
  {$\mathrm{GL}(3)$} {$L$}-Functions}}, Annals of Mathematics \textbf{173}
  (2011), no.~1, 301--336. \MR{2753605}

\bibitem[Mil06]{Miller2006}
Stephen~David Miller,
  \emph{\href{https://doi.org/10.1353/ajm.2006.0027}{Cancellation in Additively
  Twisted Sums on {$\mathrm{GL}(n)$}}}, American Journal of Mathematics
  \textbf{128} (2006), no.~3, 699--729. \MR{2230922}

\bibitem[MS06]{MillerSchmid2006}
Stephen~David Miller and Wilfried Schmid,
  \emph{\href{https://doi.org/10.4007/annals.2006.164.423}{Automorphic
  Distributions, {$L$}-Functions, and {V}oronoi Summation for
  {$\mathrm{GL}(3)$}}}, Annals of Mathematics \textbf{164} (2006), no.~2,
  423--488. \MR{2247965}

\bibitem[Raj12]{Rajkumar2012}
Krishnan Rajkumar,
  \emph{\href{https://www.imsc.res.in/xmlui/bitstream/handle/123456789/339/HBNI\%20Th52.pdf?sequence=1\&isAllowed=y}{Zeros
  of General {$L$}-Functions on the Critical Line}}, {Ph.D.} thesis, Homi
  Bhabha National Institute, 2012.

\bibitem[RY15a]{RenYe2015-2}
Xiumin Ren and Yangbo Ye,
  \emph{\href{https://doi.org/10.1007/s11425-014-4955-3}{Resonance and Rapid
  Decay of Exponential Sums of {F}ourier Coefficients of a {M}aass Form for
  {$\mathrm{GL}_{m}(\mathbb{Z})$}}}, Science China. Mathematics \textbf{58}
  (2015), no.~10, 2105--2124. \MR{3400638}

\bibitem[RY15b]{RenYe2015}
\bysame, \emph{\href{https://doi.org/10.1090/S0002-9947-2014-06208-9}{Resonance
  of Automorphic Forms for {$GL(3)$}}}, Transactions of the American
  Mathematical Society \textbf{367} (2015), no.~3, 2137--2157. \MR{3286510}

\bibitem[Sel89]{Selberg1989}
Atle Selberg,
  \emph{\href{https://link.springer.com/book/9783642410215}{Collected Papers.
  {V}ol. {I}}}, Springer-Verlag, Berlin, 1989, With a Foreword by Komaravolu
  Chandrasekharan. \MR{1117906}

\bibitem[Tit51]{Titchmarsh1987}
Edward~Charles Titchmarsh,
  \emph{\href{https://global.oup.com/academic/product/the-theory-of-the-riemann-zeta-function-9780198533696?cc=jp\&lang=en\&}{The
  Theory of the {R}iemann Zeta-Function}}, second ed., Oxford, at the Clarendon
  Press, 1951, Revised by Roger Heath-Brown. \MR{46485}

\end{thebibliography}

\providecommand{\bysame}{\leavevmode\hbox to3em{\hrulefill}\thinspace}
\providecommand{\MR}{\relax\ifhmode\unskip\space\fi MR }
\providecommand{\MRhref}[2]{%
  \href{http://www.ams.org/mathscinet-getitem?mr=#1}{#2}
}
\providecommand{\href}[2]{#2}

\end{document}